\documentclass[12pt,reqno]{amsart}

\usepackage{amsmath,amssymb,ifthen}
\usepackage{fullpage}
\usepackage{graphicx,psfrag,subfigure}
\usepackage{color}
\usepackage{moreverb}
\usepackage{amsthm}
\usepackage{mathrsfs}
\usepackage{hhline}
\usepackage{bbm}
\usepackage[english]{babel}
\usepackage{tikz}
\usepackage[utf8]{inputenc}
\usepackage{amsfonts}
\usepackage{enumerate}
\usepackage{hyperref}

\newcommand{\diam}{{\rm diam}}
\newcommand{\dist}{{\rm dist}}

\newcommand{\const}[1]{C_{\text{\rm#1}}}

\newcommand{\set}[2]{\big\{#1\,:\,#2\big\}}

\newcommand{\dual}[3][]{#1\langle#2\,,\,#3#1\rangle}
\newcommand{\norm}[3][]{#1\|#2#1\|_{#3}}

\newcommand\K{\mathbb{K}}
\newcommand\N{\mathbb{N}}
\newcommand\R{\mathbb{R}}

\newcommand\MM{\mathcal M}

\newcommand\OO{{\mathcal O}}
\newcommand\PP{\mathcal P}

\newcommand\TT{\mathcal T}

\newcommand\XX{\mathcal X}

\numberwithin{equation}{section}
\numberwithin{figure}{section}
\newtheorem{theorem}{Theorem}[section]
\newtheorem{proposition}[theorem]{Proposition}
\newtheorem{lemma}[theorem]{Lemma}
\newtheorem{corollary}[theorem]{Corollary}
\newtheorem{algorithm}[theorem]{Algorithm}

\newtheorem{remark}[theorem]{Remark}

\newcommand*\patchAmsMathEnvironmentForLineno[1]{%
  \expandafter\let\csname old#1\expandafter\endcsname\csname #1\endcsname
  \expandafter\let\csname oldend#1\expandafter\endcsname\csname end#1\endcsname
  \renewenvironment{#1}%
     {\linenomath\csname old#1\endcsname}%
     {\csname oldend#1\endcsname\endlinenomath}}%
\newcommand*\patchBothAmsMathEnvironmentsForLineno[1]{%
  \patchAmsMathEnvironmentForLineno{#1}%
  \patchAmsMathEnvironmentForLineno{#1*}}%
\AtBeginDocument{%
\patchBothAmsMathEnvironmentsForLineno{equation}%
\patchBothAmsMathEnvironmentsForLineno{align}%
\patchBothAmsMathEnvironmentsForLineno{flalign}%
\patchBothAmsMathEnvironmentsForLineno{alignat}%
\patchBothAmsMathEnvironmentsForLineno{gather}%
\patchBothAmsMathEnvironmentsForLineno{multline}%
}
\usepackage[mathlines]{lineno}

\usepackage{bbm} 
\usepackage{caption}
\usepackage{soul}

\renewcommand{\d}{\,{\rm d}}

\newcommand{\n}{{\bf n}}
\newcommand{\x}{{\bf x}}
\newcommand{\y}{{\bf y}}

\newcommand{\1}{\mathbbm{1}}

\DeclareMathOperator{\Ei}{Ei}

\title{Adaptive space-time BEM for the heat equation}

\author{Gregor Gantner}
\address{Korteweg-de Vries Institute for Mathematics\\
University of Amsterdam, P.O. Box 94248, 1090 GE Amsterdam, The Netherlands}
\email{G.Gantner@uva.nl}%
\author{Raymond van Venetië}
\address{Korteweg-de Vries Institute for Mathematics\\
University of Amsterdam, P.O. Box 94248, 1090 GE Amsterdam, The Netherlands}
\email{R.vanVenetie@uva.nl}

\begin{document}

\begin{abstract}
We consider the space-time boundary element method (BEM) for the heat equation with prescribed initial and Dirichlet data.
We propose a residual-type {\slshape a posteriori} error estimator that is a lower bound and, up to weighted $L_2$-norms of the residual, also an upper bound for the unknown BEM error.
The possibly locally refined meshes are assumed to be prismatic, i.e., their elements are tensor-products $J\times K$ of elements in time $J$ and space $K$.
While the results do not depend on the local aspect ratio between time and space, assuming the scaling $|J| \eqsim \diam(K)^2$ for all elements and using Galerkin BEM,
the estimator is shown to be efficient and reliable without the additional $L_2$-terms.
In the considered numerical experiments on two-dimensional domains in space, the estimator seems to be equivalent to the error, independently of these assumptions.
In particular for adaptive anisotropic refinement, both converge with the best possible convergence rate.
\end{abstract}

\date{\today}
\maketitle


\section{Introduction}
In the last years, there has been a growing interest in simultaneous space-time boundary element methods (BEM) for the heat equation~\cite{cs13,mst14,mst15,ht18,cr19,dns19,dzomk19,tausch19,zwom21}.
In contrast to the differential operator based variational formulation on the space-time cylinder, the variational formulation corresponding to space-time BEM is coercive~\cite{an87,costabel90} so that the discretized version always has a unique solution  regardless of the chosen trial space which is even quasi-optimal in the natural energy norm.
Moreover, it is naturally applicable on unbounded domains and only requires a mesh of the lateral boundary of the space-time cylinder resulting in a dimension reduction.
The potential disadvantage that discretizations lead to dense matrices due to the nonlocality of the boundary integral operators has been tackled, e.g., in \cite{mst14,mst15,ht18} via the fast multipole method and $\mathcal{H}$-matrices.

Two often mentioned advantages of simultaneous space-time methods are their potential for massive parallelization as well as their potential for fully adaptive refinement to resolve singularities local in both space and time.
While the first advantage has been investigated in, e.g., \cite{dzomk19,zwom21}, the latter requires suitable {\slshape a posteriori} computable error estimators, which have not been developed yet for the heat equation.
Indeed, concerning {\slshape a posteriori} error estimation as well as adaptive refinement for BEM for time-dependent problems, we are only aware of the works \cite{glaefke12,goss20} for the wave equation in two and three space dimensions, respectively.

In the present manuscript, we generalize the results~\cite{faermann00,faermann02} from Faermann for stationary PDEs to the heat equation:
Let $\Omega\subset\R^{n}$, $n=2,3$, be a Lipschitz domain with boundary $\Gamma:=\partial\Omega$ and $T>0$ a given end time point with corresponding time interval $I:=(0,T)$.
We abbreviate the space-time cylinder $Q:=I\times\Omega$ with lateral boundary $\Sigma:=I\times\Gamma$ and corresponding outer normal vector $\n\in\R^{n}$.
With the heat kernel
\begin{align*}
 G(t,\x) := \begin{cases} \frac{1}{(4\pi t)^{n/2}} \, e^{-\frac{|\x|^2}{4t}} \quad &\text{for }(t,\x)\in (0,\infty)\times\R^{n},
 \\
 0 \quad &\text{else},
 \end{cases}
\end{align*}
and a given function $f:\Sigma\to\R$, we consider the boundary integral equation
\begin{align}\label{eq:single layer operator}
 (\mathscr{V} \phi)(t,\x) := \int_\Sigma G(t-s,\x-\y) \phi(t-s,\x-\y) \d \y \d s = f(t,\x) \quad\text{for a.e.\ }(t,\x)\in \Sigma.
\end{align}
Here, $\mathscr{V}$ is the single-layer operator.
For given initial condition $u_0:\Omega\to\R$ and Dirichlet data $u_D:\Sigma\to\R$, such equations arise from the 
heat equation
\begin{align}\label{eq:interior}
\begin{array}{rcll}
\partial_t u - \Delta u & = & 0 & \text{ on } Q,\\
u & = & u_D & \text{ on }\Sigma,\\
u(0,\cdot) & = & u_0 & \text{ on }\Omega.
\end{array}
\end{align}

Let $\PP$ be a mesh of the space-time boundary $\Sigma$ consisting of prismatic elements $J\times K$ with $J\subseteq \overline I$ and $K\subseteq \Gamma$, and let $\Phi$ be an associated approximation of $\phi$.
Typically, $\Phi$ is a piecewise polynomial with respect to~$\PP$.
As $\mathscr{V}$ is an isomorphism from the dual space $H^{-1/2,-1/4}(\Sigma):= H^{1/2,1/4}(\Sigma)'$  to the anisotropic Sobolev space $H^{1/2,1/4}(\Sigma)$, the discretization error $\norm{\phi-\Phi}{H^{-1/2,-1/4}(\Sigma)}$ is equivalent to the norm of the residual $\norm{f-\mathscr{V}\Phi}{H^{1/2,1/4}(\Sigma)}$.
We show that the residual norm can be localized up to weighted $L_2$-terms, i.e.,
\begin{align*}
 \sum_{J\times K\in\PP} \eta_\PP(\Phi,J\times K)^2
 \lesssim \norm{f-\mathscr{V}\Phi}{H^{1/2,1/4}(\Sigma)}^2
 \lesssim \sum_{J\times K\in\PP} \eta_\PP(\Phi,J\times K)^2 + \zeta_\PP(\Phi,J\times K)^2,
\end{align*}
where $\eta_\PP(\Phi,J\times K)^2$ measures the $H^{1/2,1/4}$-seminorm of the residual in a neighborhood of $J\times K$ and $\zeta_\PP(\Phi) := (\diam(K)^{-1} + |J|^{-1/2})\norm{f-\mathscr{V} \Phi}{L_2(J\times K)}^2$.
The hidden constants depend only on the regularity of the 
of the meshes found by fixing either the temporal or the spatial coordinate in $\PP$.
In particular, we do not require any assumption on the relation between the spatial and temporal size of the mesh elements, making anisotropically refined meshes possible.

If the elements satisfy the scaling $|J| \eqsim \diam(K)^{2}$ and if $\Phi$ is the Galerkin approximation of $\phi$ in a discrete space $\XX$ that contains at least all $\PP$-piecewise constant functions, then we can additionally prove that
\begin{align*}
 \zeta_\PP(\Phi,J\times K)\lesssim\eta_\PP(\Phi,J\times K).
\end{align*}
Indeed, numerical experiments (with $n=2$) suggest that this is not the case in general:
If the scaling condition is not enforced, we observe situations where the weighted $L_2$-terms $\zeta$ do not decay under mesh-refinement.

That being said, the estimator $\eta$ does not only behave efficiently but also reliably in all considered examples.
Moreover, anisotropic refinement steered by the space- and time-components of the estimator always yield the optimal algebraic convergence rate of both the estimator and the error.
The source code that we used to generate the numerical results is available at~\cite{gvv21}.

\subsection*{Outline}
The remainder of this work is organized as follows:
Section~\ref{sec:preliminaries} summarizes the general principles of the space-time boundary element method for the heat equation.
Section~\ref{sec:posteriori} recalls the localization argument of \cite{faermann00,faermann02} and applies it to anisotropic Sobolev spaces (Theorem~\ref{thm:localization}).
This result is then invoked in Corollary~\ref{cor:main} for the residual, resulting in efficient and reliable {\slshape a posteriori} computable error bounds.
In particular, a Poincar\'e-type inequality (Lemma~\ref{lem:poincare}) allows to estimate the weighted $L_2$-terms that are still present in the upper bound from Theorem~\ref{thm:localization}.
Finally, Section~\ref{sec:numerics} introduces an adaptive algorithm for $n=2$ which is based on the derived error estimator.
Different marking and refinement strategies are presented.
The adaptive algorithm is subsequently applied to several concrete examples with typical singularities in space and time.
The stable implementation is discussed in Appendix~\ref{sec:computation}.

\section{Preliminaries}\label{sec:preliminaries}

\subsection{General notation}
Throughout and without any ambiguity, $|\cdot|$ denotes the absolute value of scalars, the Euclidean norm of vectors in $\R^m$, or the the measure of a set in $\R^m$, e.g., the length of an interval or the area of a surface in $\R^3$.
We write $A\lesssim B$ to abbreviate $A\le CB$ with some generic constant $C>0$, which is clear from the context.
Moreover, $A\eqsim B$ abbreviates $A\lesssim B\lesssim A$.

\subsection{Anisotropic Sobolev spaces}\label{sec:anisotropic spaces}
For $n$-dimensional $\omega\subseteq\Omega$ or $(n-1)$-dimensional $\omega\subseteq\Gamma$, and $\mu\in(0,1]$, we first recall the Sobolev space
\begin{align*}
 H^{\mu}(\omega):=\set{v\in L_2(\omega)}{\norm{v}{H^{\mu}(\omega)} < \infty}
\end{align*}
associated with the Sobolev--Slobodeckij norm
\begin{align*}
 \norm{v}{H^{\mu}(\omega)}^2
 := \norm{v}{L_2(\omega)}^2 + |v|_{H^{\mu}(\omega)}^2,
 \quad  |v|_{H^{\mu}(\omega)}^2
 :=\begin{cases}
 \int_\omega\int_\omega \frac{|v(\x)-v(\y)|^2}{|\x-\y|^{{\rm dim}(\omega)+2\mu}}\d\y\d\x &\text{ if }\mu\in(0,1),
 \\
 \norm{\nabla_\omega v}{L_2(\omega)}^2 \quad&\text{ if }\mu=1,
 \end{cases}
\end{align*}
where ${\rm dim}(\omega)$ denotes the dimension of $\omega$, i.e.,  $n$ or $n-1$, and $\nabla_\omega$ denotes the (weak) gradient on $\omega$, i.e., the standard gradient or the surface gradient.

Moreover, we define for any subinterval $J\subseteq\overline I$, $\nu\in(0,1]$, and any Banach space $X$,
\begin{align*}
 H^\nu(J;X) := \set{v\in L_2(J;X)}{\norm{v}{H^{\nu}(J;X)} < \infty}
\end{align*}
associated with the norm
\begin{align*}
 \norm{v}{H^{\nu}(J;X)}^2
 := \norm{v}{L_2(J;X)}^2 + |v|_{H^{\nu}(J;X)}^2,
 \quad  |v|_{H^{\nu}(J;X)}^2
 :=\begin{cases}
 \int_J\int_J \frac{\norm{v(t)-v(s)}{X}^2}{|t-s|^{1+2\nu}}\d s\d t &\text{ if }\nu\in(0,1),
 \\
 \norm{\partial_t v}{L_2(\omega)}^2 &\text{ if }\nu=1,
 \end{cases}
\end{align*}
where $\partial_t$ denotes the (weak) time derivative.
If $X=\R$, we simply write $H^\nu(J)$, $\norm{v}{H^\nu(J)}$, and $|v|_{H^\nu(J)}$.
Finally, we recall the anisotropic Sobolev space
\begin{align*}
 H^{\mu,\nu}(J\times\omega) := L_2(J;H^{\mu}(\omega)) \cap H^{\nu}(J;L_2(\omega))
\end{align*}
with corresponding norm
\begin{align*}
 \norm{v}{H^{\mu,\nu}(J\times\omega)}^2
 := \norm{v}{L_2(J;H^{\mu}(\omega))}^2 + \norm{v}{H^{\nu}(J;L_2(\omega))}^2
 \quad\text{for all } v\in H^{\mu,\nu}(J\times\omega).
\end{align*}
We will sometimes use the abbreviation
\begin{align*}
  |v|_{L_2(J;H^{\mu}(\omega))}^2 := \int_J |v(t,\cdot)|_{H^\mu(\omega)}^2 \d t
   \quad\text{for all } v\in L_2(J;H^\mu(\omega)).
\end{align*}

For $\omega\in\{\Omega,\Gamma\}$, we denote by $H^{-\mu,-\nu}(I\times\omega)$ the dual space of $H^{\mu,\nu}(I\times\omega)$ with duality pairing $\dual{\cdot}{\cdot}_{I\times\omega}$.
We interpret $L_2(I\times\omega)$ as subspace of $H^{-\mu,-\nu}(I\times\omega)$ via
\begin{align*}
 \dual{v}{\psi}_{I\times\omega} := \int_I\int_\omega v(t,\x) \psi(t,\x) \d\x \d t
 \quad\text{for all }v\in H^{\mu,\nu}(I\times\omega) \text{ and }\psi\in L_2(I\times\omega).
\end{align*}

\subsection{Boundary integral equations}\label{sec:integral equations}
It is well-known that for $u_0\in L^2(\Omega)$ and $u_D\in H^{1/2,1/4}(\Sigma)$, the heat equation~\eqref{eq:interior} admits a unique solution
$u\in 
H^{1,1/2}(Q)$.
With the normal derivative $\phi_N:=\partial_{\n} u \in H^{-1/2,-1/4}(\Sigma)$, $u$ satisfies the representation formula
\begin{align}\label{eq:representation}
 u = \widetilde {\mathscr M}_0 u_0 + \widetilde {\mathscr V}\phi_N  - \widetilde {\mathscr K}u_D,
\end{align}
where
\begin{align}\label{eq:initial potential}
 (\widetilde {\mathscr M}_0 u_0)(t,\x)&:=\int_\Omega G(t,\x-\y) u_0(\y) \d\y \quad\text{for all }(t,\x)\in Q
\intertext{denotes the initial potential,}
 \label{eq:single potential}
 (\widetilde {\mathscr V} \phi_N)(t,\x)&:=\int_\Sigma G(t-s,\x-\y) \phi_N(\y) \d\y \d s \quad\text{for all }(t,\x)\in Q
\intertext{denotes the single-layer potential, and }
 \label{eq:double potential}
 (\widetilde {\mathscr K} u_D)(t,\x)&:=\int_\Sigma \partial_{\n(\y)}G(t-s,\x-\y) u_D(\y) \d\y \d s \quad\text{for all }(t,\x)\in Q
\end{align}
denotes the double-layer potential.
These linear operators satisfy the mapping properties $\widetilde {\mathscr M}_0:L^2(\Omega)\to H^{1,1/2}(Q)$, $\widetilde {\mathscr V}_0:H^{-1/2,-1/4}(\Sigma)\to H^{1,1/2}(Q)$, and $\widetilde {\mathscr K}_0:H^{1/2,1/4}(\Sigma)\to H^{1,1/2}(Q)$.
The lateral trace $(\cdot)|_\Sigma$ of these potentials is given by
\begin{align*}
 (\widetilde {\mathscr M}_0 u_0)|_\Sigma = {\mathscr M}_0 u_0, \quad
 (\widetilde {\mathscr V} \phi_N)|_\Sigma = {\mathscr V} \phi_N, \quad
 (\widetilde {\mathscr K} u_D)|_\Sigma = ({\mathscr K}-1/2) u_D,
\end{align*}
where the initial operator ${\mathscr M}_0$, the single-layer operator ${\mathscr V}$, and the double-layer operator ${\mathscr K}$ are defined as in \eqref{eq:initial potential}--\eqref{eq:double potential} for $(t,\x)\in\Sigma$.
Applying the lateral trace to \eqref{eq:representation} thus results in
\begin{align}\label{eq:direct}
 {\mathscr V}\phi_N =  ({\mathscr K}+1/2) u_D - {\mathscr M}_0 u_0 ,
\end{align}
i.e., \eqref{eq:single layer operator} with $f:= ({\mathscr K}+1/2) u_D - {\mathscr M}_0 u_0$.
As the single-layer operator $\mathscr{V}$ is also coercive, i.e.,
\begin{align}\label{eq:coercivity}
 \dual{\mathscr{V}\psi}{\psi}_\Sigma \ge c_{\rm coe} \norm{\psi}{H^{-1/2,-1/4}(\Sigma)}^2 \quad \text{for all }\psi\in H^{-1/2,-1/4}(\Sigma)
\end{align}
with some constant $c_{\rm coe}>0$, \eqref{eq:direct} is uniquely solvable and the solution $\phi_N$ is just the missing normal derivative $\partial_\n u$ to compute $u$ via the representation formula~\eqref{eq:representation}.

Alternatively, one can make the ansatz $u=\widetilde{\mathscr M}_0 u_0 + \widetilde{\mathscr V}\phi$.
Indeed, both $\widetilde{\mathscr M}_0 u_0$ and $\widetilde{\mathscr V}\phi$ satisfy the heat equation, where $\widetilde{\mathscr M}_0u_0$ restricted to $\{0\}\times\Omega$ coincides with $u_0$ and  $\widetilde{\mathscr V}\phi$ vanishes there.
To satisfy the Dirichlet boundary conditions, one has to solve
\begin{align}\label{eq:indirect}
 \mathscr{V}\phi =  u_D - \mathscr{M}_0 u_0,
\end{align}
i.e., \eqref{eq:single layer operator} with $f:=u_D - \mathscr{M}_0 u_0$.
While \eqref{eq:direct} is called direct method as it directly provides the physically relevant quantity $\phi_N=\partial_\n u$, \eqref{eq:indirect} is called indirect method.

For more details and proofs, we refer to the seminal works \cite{an87,noon88,costabel90}, which considered $u_0=0$, and to \cite{dns19,dohr19} for the general case.

\subsection{Boundary meshes}
Throughout this work, we consider prismatic meshes $\PP$ of $\Sigma$:
\begin{itemize}
\item $\PP$ is a finite set of prisms of the form $P=J\times K$, where $J\subseteq \overline I =[0,T]$ is some non-empty compact interval and $K\subseteq\Gamma$ is the image of some compact Lipschitz domain\footnote{A compact Lipschitz domain is the closure of a bounded Lipschitz domain. For $n = 2$, it is a compact interval with non-empty interior.}
$\hat K\subset\R^{n-1}$ under some bi-Lipschitz mapping;
\item for all $P,\tilde{P}\in\PP$ with $P\neq \tilde{P}$, the intersection has measure zero on $\Sigma$;
\item $\PP$ is a partition of $\Sigma$, i.e., $\Sigma = \bigcup_{P\in\PP} P$.
\end{itemize}

For arbitrary $t\in\overline I$ and $\x\in\Gamma$,
we abbreviate the induced sets
\begin{align*}
 \PP|_t:=\set{K\subseteq\Gamma}{(\{t\}\times\Gamma) \cap (J\times K)\neq\emptyset \text{ for some }J\times K\in\PP}
\end{align*}
and
\begin{align*}
  \PP|_\x:=\set{J\subseteq\overline I}{(\overline I\times\{\x\}) \cap (J\times K)\neq\emptyset \text{ for some }J\times K\in\PP};
\end{align*}
see Figure~\ref{fig:mesh2d} for a visualization.
For almost all $t\in\overline I$,  $\PP|_t$ is a mesh of $\Gamma$, i.e., a partition of $\Gamma$ into finitely many compact Lipschitz domains such that the intersection of two distinct elements has measure zero on $\Gamma$.
Similarly, for almost all $\x\in\Gamma$, $\PP|_\x$ is a mesh of $\overline I$, i.e., a partition of $\overline I$ into finitely many non-empty compact intervals such that the intersection of two different intervals is at most a point.

\begin{figure}
\begin{tikzpicture}[scale=6.5]
\coordinate (O) at (0,0);
\draw[color = white, fill=blue!20]  (0.5,0) -- (0.75,0) -- (0.75,0.5) -- (0.5,0.5) -- cycle;
\draw[color = white, fill=blue!40]  (0.5,0.125) -- (0.75,0.125) -- (0.75,0.25) -- (0.5,0.25) -- cycle;
\draw[color = white, fill=white]  (0.625,0) -- (0.75,0) -- (0.75,0.0625) -- (0.625,0.0625) -- cycle;
\node at (0.625,0.1875) {$J_2\times K_2$};

\draw[color = white, fill=red!20]  (0,0.5) -- (1,0.5) -- (1,0.75) -- (0,0.75) -- cycle;
\draw[color = white, fill=red!40]  (0.25,0.5) -- (0.5,0.5) -- (0.5,0.75) -- (0.25,0.75) -- cycle;
\draw[color = white, fill=white]  (0,0.5) -- (0.125,0.5) -- (0.125,0.625) -- (0,0.625) -- cycle;
\node at (0.375,0.625) {$J_1\times K_1$};

\draw[line width=0.5mm, color=black] (O) -- (1,0);
\draw[line width=0.5mm, color=black] (1,0) -- (1,1);
\draw[line width=0.5mm, color=black] (1,1) -- (0,1);
\draw[line width=0.5mm, color=black] (0,1) -- (O);
\draw[line width=0.5mm, color=black] (0.5,0) -- (0.5,1);
\draw[line width=0.5mm, color=black] (0,0.5) -- (1,0.5);

\draw[line width=0.5mm, color=black] (0.25,0) -- (0.25,0.5);
\draw[line width=0.5mm, color=black] (0,0.25) -- (0.5,0.25);
\draw[line width=0.5mm, color=black] (0.125,0) -- (0.125,0.5);
\draw[line width=0.5mm, color=black] (0.0625,0.25) -- (0.0625,0.5);

\draw[line width=0.5mm, color=black] (0,0.75) -- (0.5,0.75);
\draw[line width=0.5mm, color=black] (0.25,0.5) -- (0.25,0.75);
\draw[line width=0.5mm, color=black] (0,0.625) -- (0.25,0.625);
\draw[line width=0.5mm, color=black] (0.125,0.5) -- (0.125,0.625);
\draw[line width=0.5mm, color=black] (0.75,0.0625) -- (1,0.0625);

\draw[line width=0.5mm, color=black] (0.5,0.25) -- (1,0.25);
\draw[line width=0.5mm, color=black] (0.5,0.125) -- (1,0.125);
\draw[line width=0.5mm, color=black] (0.75,0) -- (0.75,0.25);
\draw[line width=0.5mm, color=black] (0.625,0) -- (0.625,0.125);
\draw[line width=0.5mm, color=black] (0.625,0) -- (0.625,0.125);
\draw[line width=0.5mm, color=black] (0.625,0.0625) -- (0.75,0.0625);

\draw[line width=0.25mm, color=blue,dashed] (0.875,0) -- (0.875,1);
\draw[line width=0.25mm,color=blue] plot[only marks,mark=x,mark size=0.5] coordinates{(0.875,0) (0.875,0.0625) (0.875,0.125) (0.875,0.25) (0.875,0.5) (0.875,1)};
\node[below] at (0.875,0) {$\x=\tfrac{7}{8}$};
\draw[line width=0.25mm, color=red,dashed] (0,0.375) -- (1,0.375);
\draw[line width=0.25mm,color=red] plot[only marks,mark=x,mark size=0.5] coordinates{(0,0.375) (0.0625,0.375) (0.125,0.375)  (0.25,0.375) (0.5,0.375) (1,0.375)};
\node[left] at (0,0.375) {$t=\tfrac{3}{8}$};
\end{tikzpicture}

\caption{Prismatic mesh $\PP$ for $\Gamma=[0,1]$ and $T=1$.
The dashed blue and red lines indicate the meshes $\PP|_t$ for $t=\tfrac38$ and $\PP|_\x$ for $\x=\tfrac{7}{8}$, respectively, where the corresponding elements are limited by crosses.
For the elements $J_1\times K_1=[\tfrac12,\tfrac34]\times [\tfrac18,\tfrac14]$ and  $J_2\times K_2=[\tfrac14,\tfrac12]\times[\tfrac12,\tfrac34]$, the  integration domains
$\bigcup_{{\tilde{J}\times\tilde{K}\in\PP \atop |J_1\cap \tilde{J}|>0} \atop K_1\cap\tilde{K} \neq \emptyset} (J_1\cap \tilde J) \times (K_1\cup \tilde K)$
and
$\bigcup_{{\tilde{J}\times\tilde{K}\in\PP \atop J_2\cap \tilde{J} \neq \emptyset} \atop |K_2\cap\tilde{K}|>0} (J_2\cup \tilde J) \times (K_2\cap \tilde K)$
from \eqref{eq:efficiency} are highlighted in (light) red and (light) blue, respectively.}
\label{fig:mesh2d}
\end{figure}

Note that for one fixed prismatic mesh $\PP$ there exist constants $\const{nei}\ge1$, $\const{dist}\ge1$, $\const{shape}\ge1$, and $\const{lqu}\ge1$ such that:
\begin{itemize}
\item for almost all $t\in\overline I$,
the number of neighbors of an element in $\PP|_t$ is bounded, i.e.,
\begin{align}\label{eq:neighbors}
 \#\set{\tilde{K}\in\PP|_t}{K\cap\tilde{K}\neq\emptyset} \le \const{nei} \quad\text{for all }K\in\PP|_t.
\end{align}
\item for almost all $t\in\overline I$,
    the elements of $\PP|_t$ are uniformly away from non-neighboring elements, i.e.,
\begin{align}\label{eq:dist}
 \diam(K)\le \const{dist} \dist(K,\tilde{K}) \quad \text{for all } K,\tilde{K}\in\PP|_t\text{ with }K\cap\tilde{K}=\emptyset;
\end{align}
\item for almost all $t\in\overline I$,
the elements of $\PP|_t$ are shape-regular, i.e.,
\begin{align}\label{eq:shape}
 \const{shape}^{-1} |K|^{n-1} \le  \diam(K)^{n-1}\le \const{shape} |K| \quad \text{for all } K \in\PP|_t;
\end{align}
\item for almost all $\x\in\Gamma$, $\PP|_\x$ is locally quasi-uniform, i.e.,
\begin{align}\label{eq:local quasi-uniformity}
 |J| \le \const{lqu}|\tilde{J}|\quad\text{for all }J,\tilde{J}\in\PP|_\x\text{ with }J\cap \tilde{J}\neq\emptyset.
\end{align}
%
\end{itemize}

In the remainder of this work, we will always indicate the dependence of estimates on these particular constants.

\begin{remark}
If, for $n=2$, the meshes $\PP|_t$ are found by iteratively bisecting some initial mesh and the level difference of neighboring elements is bounded by $1$, then the constants from \eqref{eq:neighbors}--\eqref{eq:shape} depend only on the initial mesh;   cf.~\cite{affkp13}.
For $n=3$, the same holds true if the initial mesh is for instance a conforming (curvilinear) triangulation of $\Gamma$ and one iteratively applies newest vertex bisection.
The arguments for \eqref{eq:neighbors}--\eqref{eq:dist} are found in \cite[Section~2.3 and 4.1]{affkmp17}.
\end{remark}

\subsection{Boundary element method}
Given a prismatic boundary mesh $\PP$ and an associated finite-dimensional trial space $\XX\subset H^{-1/2,-1/4}(\Sigma)$, e.g., the space of all $\PP$-piecewise polynomials of some fixed degree in space and time, let $\Phi\in\XX$ denote the Galerkin discretization of the solution $\phi$ of the boundary integral equation~\eqref{eq:single layer operator}, i.e.,
\begin{align}\label{eq:Galerkin}
 \dual{\mathscr{V}\Phi}{\Psi}_\Sigma = \dual{f}{\Psi}_\Sigma
 \quad \text{for all }\Psi\in \XX,
\end{align}
which is equivalent to the Galerkin orthogonality
\begin{align}\label{eq:orthogonality}
 \dual{\mathscr{V}(\phi-\Phi)}{\Psi}_\Sigma = 0
 \quad \text{for all }\Psi\in \XX.
\end{align}

Note that coercivity~\eqref{eq:coercivity} guarantees unique solvability of the latter equations, and the C\'ea lemma applies
\begin{align}\label{eq:cea}
 \norm{\phi-\Phi}{H^{-1/2,-1/4}(\Sigma)}
 \le
 \frac{\const{cont}}{c_{\rm coe}}\,\min_{\Psi\in\XX} \norm{\phi-\Psi}{H^{-1/2,-1/4}(\Sigma)},
\end{align}
where $\const{cont}$ is the operator norm of $\mathscr{V}:H^{-1/2,-1/4}(\Sigma) \to H^{1/2,1/4}(\Sigma)$.
Suppose $\PP = \set{J\times K}{J\in \PP_{\overline I}, K\in\PP_\Gamma}$ is a full tensor-mesh corresponding to a mesh $\PP_\Gamma$ of $\Gamma$ with uniform mesh-size $h_\x\eqsim \diam(K)$ for all $K\in\PP_\Gamma$ and a mesh $\PP_{\overline I}$ of $\overline I$ with uniform step-size $h_t\eqsim h_\x^\sigma$ for some $\sigma>0$.
Using $\PP$-piecewise polynomials of some degree $p_\x\in\N_0$ in space- and some degree $p_t\in\N_0$ in time-direction as trial space $\XX$, then gives the error decay rate
\begin{align}\label{eq:rates}
 \min_{\Psi\in\XX} \norm{\phi-\Psi}{H^{-1/2,-1/4}(\Sigma)} \lesssim N^{-\frac{\min\{p_\x+3/2,(p_t+5/4)\sigma\}}{n-1+\sigma}}
 \quad\text{for all smooth }\phi;
\end{align}
see \cite[Theorem~3.3]{cr19}.
Here, $N\eqsim h_\x^{-(n-1)} h_t^{-1} = h_\x^{n-1+\sigma}$ denotes the number of degrees of freedom.
The optimal grading parameter is thus given by $\sigma = (p_\x+\tfrac32)/(p_t+\tfrac54)$ with resulting rate $\OO\big(N^{-\frac{p_\x+3/2}{n-1+\sigma}}\big)$.

\section{A posteriori error estimation}\label{sec:posteriori}

As $\mathscr V$ is an isomorphism, it holds that
\begin{align}\label{eq:residual norm}
 \norm{\phi-\Phi}{H^{-1/2,-1/4}(\Sigma)} \eqsim  \norm{f-\mathscr V\Phi}{H^{1/2,1/4}(\Sigma)}.  
\end{align}
Here, $\Phi\in H^{-1/2,-1/4}(\Sigma)$ can be an arbitrary approximation of the solution $\phi$ of~\eqref{eq:single layer operator}.
While the right-hand side is in principle {\slshape a posteriori} computable, the computation of the Sobolev--Slobodeckij norm over the full space-time boundary $\Sigma$ is expensive, and it does not provide any information on where to locally refine the given mesh to increase the accuracy of the approximation.
According to~\eqref{eq:residual norm}, it is sufficient to derive suitable estimate for the residual $f-\mathscr V\Phi$ in the $H^{1/2,1/4}(\Sigma)$-norm.
Recall that this term is $L_2(\Sigma)$-orthogonal to all functions $\Psi\in\XX$ provided that $\Phi$ is the Galerkin approximation of $\phi$ in $\XX$; see~\eqref{eq:orthogonality}.

\subsection{Localization of the anisotropic Sobolev--Slobodeckij norm}
The following proposition provides the key argument for our {\slshape a posteriori} error estimation.
While the first inequality is trivial, the original version of the second one already goes back to \cite{faermann00,faermann02}.
We make use of the slightly generalized version from \cite[Lemma~4.5]{gp20}; see \cite[Lemma~5.3.2]{gantner17} for a detailed proof.

\begin{proposition}\label{prop:space-localization}
Let $\mu\in(0,1)$ and $\PP_\Gamma$ be a mesh of $\Gamma$.
Then, there exist constants $C_1,C_2>0$ such that for all $v\in H^\mu(\Gamma)$, there holds that
\begin{align}
\begin{split}
  C_1^{-1}\sum_{K\in\PP_\Gamma} \sum_{\substack{\tilde{K}\in\PP_\Gamma\\K\cap\tilde{K}\neq\emptyset}} |v|_{H^\mu(K\cup\tilde{K})}^2
  \le \norm{v}{H^\mu(\Gamma)}^2
 &\le \sum_{K\in\PP_\Gamma} \sum_{\substack{\tilde{K}\in\PP_\Gamma\\K\cap\tilde{K}\neq\emptyset}} |v|_{H^\mu(K\cup\tilde{K})}^2
 \\
 &\quad+ C_2\sum_{K\in\PP_\Gamma} \diam(K)^{-2\mu} \norm{v}{L_2(K)}^2.
\end{split}
\end{align}
The constant $C_1$ is given as $C_1=2(\const{nei}+1)^2$ with $\const{nei}$ from~\eqref{eq:neighbors} (with $\PP|_t$ replaced by $\PP_\Gamma$),
and $C_2$ depends only on the dimension $n$, $\mu$, $\Gamma$, and the constant $\const{dist}$ from~\eqref{eq:dist} (with $\PP|_t$ replaced by $\PP_\Gamma$).
\hfill$\square$
\end{proposition}

Note that local quasi-uniformity~\eqref{eq:local quasi-uniformity} (with $\PP|_\x$ replaced by $\PP_{\overline I}$) of a time mesh $\PP_{\overline I}$ is actually equivalent to
\begin{align}
  \diam(J) = |J| \le \const{lqu} \dist(J,\tilde{J}) \quad \text{for all } J,\tilde{J}\in\PP_{\overline I} \text{ with } J \cap \tilde{J}=\emptyset.
\end{align}
Moreover, for any element $J\in\PP_{\overline I}$, there are at most three $\tilde{J}\in\PP_{\overline I}$ with $J\cap \tilde{J}\neq\emptyset$.
In particular, the same reference as before applies and we also obtain the following proposition.

\begin{proposition}\label{prop:time-localization}
Let $\nu\in(0,1)$ and $\PP_{\overline I}$ be a mesh of $\overline I$.
Then, there exist constants $C_1,C_2>0$ such that for all $v\in H^\nu(I)$, there holds that
\begin{align}
  C_1^{-1}\sum_{J\in\PP_{\overline I}} \sum_{\substack{\tilde{J}\in\PP_{\overline I}\\ J\cap J\neq\emptyset}} |v|_{H^\nu(J\cup \tilde{J})}^2
  \le \norm{v}{H^\nu(I)}^2
 \le \sum_{J\in\PP_{\overline I}} \sum_{\substack{\tilde{J}\in\PP_{\overline I}\\ J\cap J\neq\emptyset}} |v|_{H^\nu(J\cup \tilde{J})}^2
 + C_2\sum_{J\in\PP_{\overline I}} |J|^{-2\nu} \norm{v}{L_2(J)}^2.
\end{align}
The constant $C_1$ is given as $C_1=32$, and $C_2$ depends only on $\nu$, $|I|$, and the constant $\const{lqu}$ from \eqref{eq:local quasi-uniformity} (with $\PP|_\x$ replaced by $\PP_{\overline I}$).
\hfill$\square$
\end{proposition}

The latter two propositions allow to derive the following {\slshape a posteriori error estimation}, which can be employed for arbitrary approximations $\Phi$.

\begin{theorem}\label{thm:localization}
Let $\mu,\nu\in(0,1)$ and $\PP$ be a prismatic mesh of $\Sigma$.
Then, there exist constants $\const{eff}', \const{rel}''>0$ such that for all $v\in H^{\mu,\nu}(\Sigma)$, there holds that
\begin{align}\label{eq:efficiency}
 \sum_{J\times K\in\PP} \Big(\sum_{{\tilde{J}\times\tilde{K}\in\PP \atop |J\cap \tilde{J}|>0} \atop K\cap\tilde{K} \neq \emptyset} |v|_{L_2(J\cap \tilde{J}; H^\mu(K\cup\tilde{K}))}^2
 + \sum_{{\tilde{J}\times\tilde{K}\in\PP \atop J\cap \tilde{J}\neq\emptyset} \atop |K\cap\tilde{K}| > 0} |v|_{H^\nu(J\cup \tilde{J}; L_2(K\cap\tilde{K}))}^2\Big)
 \le (\const{eff}')^2 \norm{v}{H^{\mu,\nu}(\Sigma)}^2
\end{align}
as well as
\begin{align}\label{eq:reliability}
\begin{split}
 (\const{rel}')^{-2}\norm{v}{H^{\mu,\nu}(\Sigma)}^2
 \le
 \sum_{J\times K\in\PP} \Big(\sum_{{\tilde{J}\times\tilde{K}\in\PP \atop |J\cap \tilde{J}|>0} \atop K\cap\tilde{K} \neq \emptyset} |v|_{L_2(J\cap \tilde{J}; H^\mu(K\cup\tilde{K}))}^2
 + \sum_{{\tilde{J}\times\tilde{K}\in\PP \atop J\cap \tilde{J}\neq\emptyset} \atop |K\cap\tilde{K}| > 0} |v|_{H^\nu(J\cup \tilde{J}; L_2(K\cap\tilde{K}))}^2\Big)
 \\
 +  \sum_{J\times K\in\PP} \big(\diam(K)^{-2\mu} + |J|^{-2\nu}\big)\norm{v}{L_2(J\times K)}^2;
\end{split}
\end{align}
see Figure~\ref{fig:mesh2d} for a visualization of the involved integration domains.
The constant $\const{eff}'$ is given as $\const{eff}'=\max(2(\const{nei}+1)^2,32)$ with $\const{nei}$ from~\eqref{eq:neighbors},
and $\const{rel}'$ depends only on $n$, $\mu$, $\nu$, $\Gamma$, $|I|$ and the constants $\const{dist}$ from~\eqref{eq:dist} as well as $\const{lqu}$ from \eqref{eq:local quasi-uniformity}.
\end{theorem}
\begin{proof}
We split the proof into four steps.

\noindent{\bf Step~1:} In this step, we bound $\norm{v}{L_2(I;H^\mu(\Gamma))}$ from below.
Proposition~\ref{prop:space-localization} gives that
\begin{align*}
 \norm{v}{L_2(I;H^\mu(\Gamma))}^2
 &=\int_I \norm{v(t,\cdot)}{H^\mu(\Gamma)}^2 \d t
 \gtrsim \int_I \sum_{K\in\PP|_t} \sum_{\substack{\tilde{K}\in\PP|_t\\K\cap\tilde{K}\neq\emptyset}} |v(t,\cdot)|_{H^\mu(K\cup\tilde{K})}^2 \d t.
\end{align*}
Note that $K\in\PP|_t$ is equivalent to $J\times K\in\PP$ for some $J$ with $t\in J$.
With the indicator function $\mathbbm{1}_S$ of a set $S$, the last term thus is equal to
\begin{align*}
 \int_I \sum_{K\in\PP|_t} \sum_{\substack{\tilde{K}\in\PP|_t\\K\cap\tilde{K}\neq\emptyset}} |v(t,\cdot)|_{H^\mu(K\cup\tilde{K})}^2 \d t
 &= \int_I \sum_{J\times K\in\PP} \mathbbm{1}_J(t) \sum_{\substack{\tilde{J}\times\tilde{K}\in\PP\\K\cap\tilde{K}\neq\emptyset}} \mathbbm{1}_{\tilde{J}}(t) |v(t,\cdot)|_{H^\mu(K\cup\tilde{K})}^2 \d t
 \\
 &= \sum_{J\times K\in\PP} \sum_{{\tilde{J}\times\tilde{K}\in\PP \atop |J\cap \tilde{J}|>0} \atop K\cap\tilde{K} \neq \emptyset} |v|_{L_2(J\cap \tilde{J}; H^\mu(K\cup\tilde{K}))}^2.
\end{align*}

\noindent{\bf Step~2:} In this step, we bound $\norm{v}{L_2(I;H^\mu(\Gamma))}$ from above.
Proposition~\ref{prop:space-localization} gives that
\begin{align}
\notag
 \norm{v}{L_2(I;H^\mu(\Gamma))}^2
 &=\int_I \norm{v(t,\cdot)}{H^\mu(\Gamma)}^2 \d t
 \\
 \label{eq:estimate ut}
 &\lesssim \int_I \sum_{K\in\PP|_t} \sum_{\substack{\tilde{K}\in\PP|_t\\K\cap\tilde{K}\neq\emptyset}} |v(t,\cdot)|_{H^\mu(K\cup\tilde{K})}^2
 +  \sum_{K\in\PP|_t} \diam(K)^{-2\mu} \norm{v(t,\cdot)}{L_2(K)}^2 \d t.
\end{align}
The first term in~\eqref{eq:estimate ut}  has already been treated in Step~1.
As $K\in\PP|_t$ is equivalent to $J\times K\in\PP$ for some $J$ with $t\in J$,
the second term reads
\begin{align*}
 \int_I\sum_{K\in\PP|_t} \diam(K)^{-2\mu} \norm{v(t,\cdot)}{L_2(K)}^2 \d t
 & =  \int_I\sum_{J\times K\in\PP} \mathbbm{1}_J(t) \, \diam(K)^{-2\mu} \norm{v(t,\cdot)}{L_2(K)}^2 \d t
 \\
 &= \sum_{J\times K\in\PP} \diam(K)^{-2\mu}\norm{v}{L_2(J\times K)}^2.
\end{align*}

\noindent{\bf Step~3:} In this step, we bound $\norm{v}{H^\nu(I;L_2(\Gamma))}$ from below.
The Fubini theorem, Proposition~\ref{prop:time-localization}, and the same argument as in Step~1 give that
\begin{align*}
 \norm{v}{H^\nu(I;L_2(\Gamma))}^2
 =\int_\Gamma \norm{v(\cdot,\x)}{H^\nu(I)}^2 \d \x
 &\gtrsim  \int_\Gamma \sum_{J\in\PP|_\x} \sum_{\substack{\tilde{J}\in\PP|_\x \\ J\cap \tilde{J}\neq\emptyset}} |v(\cdot,\x)|_{H^\nu(J\cup \tilde{J})}^2 \d\x
 \\
 &=  \int_\Gamma \sum_{J\times K\in\PP} \mathbbm{1}_{K}(\x) \sum_{\substack{\tilde{J}\times\tilde{K}\in\PP \\ J\cap \tilde{J}\neq\emptyset}} \mathbbm{1}_{\tilde{K}}(\x) |v(\cdot,\x)|_{H^\nu(J\cup \tilde{J})}^2 \d\x
 \\
 &= \sum_{J\times K\in\PP} \sum_{{\tilde{J}\times\tilde{K}\in\PP \atop J\cap \tilde{J}\neq\emptyset} \atop |K\cap\tilde{K}| > 0} |v|_{H^\nu(J\cup \tilde{J}; L_2(K\cap\tilde{K}))}^2.
\end{align*}

\noindent{\bf Step~4:} In this step, we bound $\norm{v}{H^\nu(I;L_2(\Gamma))}$ from above.
The Fubini theorem and Proposition~\ref{prop:time-localization} give that
\begin{align}
\notag
 \norm{v}{H^\nu(I;L_2(\Gamma))}^2
 &=\int_\Gamma \norm{v(\cdot,\x)}{H^\nu(I)}^2 \d \x
 \\
 \label{eq:estimate ux}
 &\lesssim \int_\Gamma \sum_{J\in\PP|_\x} \sum_{\substack{\tilde{J}\in\PP|_\x\\ J\cap \tilde{J}\neq\emptyset}} |v(\cdot,\x)|_{H^\nu(J\cup \tilde{J})}^2
 +  \sum_{J\in\PP|_\x} |J|^{-2\nu} \norm{v(\cdot,\x)}{L_2(J)}^2 \d \x.
\end{align}
The first term in~\eqref{eq:estimate ux}  has already been treated in Step~3.
The second term reads
\begin{align*}
 \int_\Gamma\sum_{J\in\PP|_\x} |J|^{-2\nu} \norm{v(\cdot,\x)}{L_2(J)}^2 \d \x
 &=  \int_\Gamma\sum_{J\times K\in\PP}\mathbbm{1}_K(\x)\, |J|^{-2\nu} \norm{v(\cdot,\x)}{L_2(J)}^2 \d \x
 \\
 &= \sum_{J\times K\in\PP} |J|^{-2\nu}\norm{v}{L_2(J\times K)}^2.
\end{align*}
This concludes the proof.
\end{proof}

\subsection{Poincar\'e-type inequality}
Assuming the grading $|J| \eqsim \diam(K)^{\mu/\nu}$ as well as $L_2(\Sigma)$-orthogonality of $v$ to piecewise constants, the following local Poincar\'e-type inequality allows to get rid of the weighted $L_2$-terms in \eqref{eq:reliability}.
The proof works essentially as in \cite[Proposition~5.3]{costabel90}, where a global version on uniform meshes is considered.

\begin{lemma}\label{lem:poincare}
Let $\mu,\nu\in(0,1)$ and $\PP$ be a prismatic mesh of $\Sigma$.
Then, there holds for all $v\in H^{\mu,\nu}(\Sigma)$ and all $J\times K \in \PP$ with $\dual{v}{1}_{L_2(J\times K)}=0$ that
\begin{align}\label{eq:poinccare}
 \norm{v}{L_2(J\times K)}^2
 \le
 \const{shape} \big(\diam(K)^{2\mu} |v|_{L_2(J; H^\mu(K))}^2 +  |J|^{2\nu} |v|_{H^\nu(J; L_2(K))}^2\big).
\end{align}
Here, $\const{shape}\ge1$ is the constant from~\eqref{eq:shape}.
\end{lemma}
\begin{proof}
Let $\Pi_J$, $\Pi_K$, and $\Pi_{J\times K}$ denote the $L_2$-orthogonal projection onto the space of constants on $J$, $K$, and $J\times K$, respectively.
Note that $\Pi_{J\times K}=\Pi_J\otimes \Pi_K$ and thus
\begin{align}
\notag
 \norm{v}{L_2(J\times K)}
 &=  \norm{(1-\Pi_{J\times K})v}{L_2(J\times K)}
 \\
 \label{eq:projections}
 &\le  \norm{(1-\Pi_{J}\otimes{\rm Id})v}{L_2(J\times K)} + \norm{(\Pi_J\otimes{\rm Id}-\Pi_{J}\otimes\Pi_K)v}{L_2(J\times K)}.
\end{align}
As $\Pi_J$ has operator norm $1$, a standard Poincaré-type inequality, see, e.g., \cite[Lemma~3.4]{faermann02} for the elementary proof, shows for the second term in \eqref{eq:projections} that
\begin{align*}
 \norm{(\Pi_J\otimes{\rm Id}-\Pi_{J}\otimes\Pi_K)v}{L_2(J\times K)}^2
 &\le  \norm{(1-{\rm Id}\otimes\Pi_K)v}{L_2(J\times K)}^2
 \\
 &=\int_J \norm{(1-\Pi_K) v(t,\cdot)}{L_2(K)}^2 \d t
 \\
 &\le \frac{\diam(K)^{n-1+2\mu}}{2|K|} \int_J |v(t,\cdot)|_{H^\mu(K)}^2 \d t
 \\
 &\le \frac{\const{shape}}{2}\, \diam(K)^{2\mu} |v|_{L_2(J;H^\mu(K)}^2.
\end{align*}
The first term in \eqref{eq:projections} can be estimated similarly
\begin{align*}
 \norm{(1-\Pi_{J}\otimes{\rm Id})v}{L_2(J\times K)}^2
 \le \frac{1}{2}|J|^{2\nu} |v|_{H^\nu(J;L_2(K))}^2,
\end{align*}
which concludes the proof.
\end{proof}

\subsection{A posteriori error estimators}\label{sec:estimators}
For arbitrary prismatic meshes $\PP$ of $\Sigma$ with some associated trial space $\XX\subset H^{-1/2,-1/4}(\Sigma)$ and $\Phi\in\XX$, we define the following error indicators for all $J\times K\in\PP$,
\begin{align*}
 \eta_{\PP}^\x(\Phi,J\times K) &:= \sum_{{\tilde{J}\times\tilde{K}\in\PP \atop |J\cap \tilde{J}|>0} \atop K\cap\tilde{K} \neq \emptyset} |f-\mathscr{V} \Phi|_{L_2(J\cap \tilde{J}; H^{1/2}(K\cup\tilde{K}))}^2,
 \\
 \eta_{\PP}^t(\Phi,J\times K)^2&:= \sum_{{\tilde{J}\times\tilde{K}\in\PP \atop J\cap \tilde{J}\neq\emptyset} \atop |K\cap\tilde{K}| > 0} |f-\mathscr{V} \Phi|_{H^{1/4}(J\cup \tilde{J}; L_2(K\cap\tilde{K}))}^2,
 \\
 \zeta_{\PP}^\x(\Phi,J\times K) &:= \diam(K)^{-1} \norm{f-\mathscr{V} \Phi}{L_2(J\times K)}^2,
 \\
 \zeta_{\PP}^t(\Phi,J\times K)^2&:= |J|^{-1/2} \norm{f-\mathscr{V} \Phi}{L_2(J\times K)}^2.
\end{align*}
The corresponding error estimators read as
\begin{align*}
 \eta_\PP(\Phi)^2 := \sum_{J\times K\in\PP}  \eta_{\PP}(\Phi,J\times K)^2
 \quad\text{with }
 \eta_{\PP}(\Phi,J\times K) ^2 := \eta_{\PP}^\x(\Phi,J\times K)^2 + \eta_{\PP}^t(\Phi,J\times K)^2,
 \\
  \zeta_\PP(\Phi)^2 := \sum_{J\times K\in\PP}  \zeta_{\PP}(\Phi,J\times K)^2
 \quad\text{with }
 \zeta_{\PP}(\Phi,J\times K) ^2 := \zeta_{\PP}^\x(\Phi,J\times K)^2 + \zeta_{\PP}^t(\Phi,J\times K)^2.
\end{align*}
With \eqref{eq:residual norm}, we overall obtain the following {\slshape a posteriori} estimates.

\begin{corollary}\label{cor:main}
Let $\phi$ be the solution of \eqref{eq:single layer operator} and $\PP$ be a prismatic mesh of $\Sigma$ with some associated discrete trial space $\XX\subset H^{-1/2,-1/4}(\Sigma)$.
Then, there exist constants $\const{eff}, \tilde C_{\rm rel}>0$ such that for arbitrary $\Phi\in\XX$, there holds that
\begin{align}\label{eq:efficiency and reliability}
 \const{eff}^{-1} \eta_\PP(\Phi)
 \le \norm{\phi-\Phi}{H^{-1/2,-1/4}(\Sigma)}
 \le \tilde C_{\rm rel} \big(\eta_\PP(\Phi)^2 + \zeta_\PP(\Phi)^2\big)^{1/2}.
\end{align}
If the space $\XX$ contains all $\PP$-piecewise constant functions and $\Phi\in\XX$ is the Galerkin approximation of $\phi$, there further holds that
\begin{align}
 \zeta_\PP(\Phi,J\times K)^2 \le \const{shape} \big(\diam(K)^{-1}+|J|^{-1/2}\big)\big(\diam(K) + |J|^{1/2}\big) \eta_{\PP}(\Phi,J\times K)^2
\end{align}
for all $J\times K\in\PP$.
If $\const{grad}^{-1}\diam(K)\le |J|^{1/2}\le \const{grad} \diam(K)$ is satisfied for all $J\times K\in\PP$ and a uniform constant $\const{grad}\ge1$, this implies the existence of a constant $\const{rel}>0$ such that
\begin{align}\label{eq:reliability2}
 \norm{\phi-\Phi}{H^{-1/2,-1/4}(\Sigma)} \le \const{rel} \eta_\PP(\Phi).
\end{align}
The constants $\const{eff}$ and $\tilde C_{\rm rel}$ are given as $\const{eff}=\const{eff}'\,\const{cont}$ and $\const{rel}=\const{rel}'/c_{\rm coe}$ with $\const{eff}'$ and $\const{rel}'$ from Theorem~\ref{thm:localization},  the operator norm $\const{cont}$ of $\mathscr V$, and $c_{\rm coe}$ from~\eqref{eq:coercivity}.
The constant $\const{rel}$ is given as $\const{rel}=\tilde C_{\rm rel}\sqrt{2\const{shape}(1+\const{grad})}$.
\qed
\end{corollary}

\begin{remark}
    According to~\eqref{eq:rates}, the required scaling $|J| \eqsim \diam(K)^2$, i.e., $\sigma=2$, is the optimal scaling for approximating smooth solutions $\phi$
    if the polynomial degrees of $\XX$ satisfy $p_\x=2p_t+1$.
\end{remark}

\section{Numerical experiments}\label{sec:numerics}

In this section, we employ the error estimator $\eta$ within an adaptive algorithm using different refinement strategies, and investigate the resulting convergence rates.
We restrict ourselves to the case $n=2$,  with $\Gamma = \partial \Omega$ being the boundary of a polygonal domain $\Omega \subset \R^2$, and set the time domain to be $I = (0,1)$.

For $\PP$ a prismatic mesh of the space-time boundary, i.e., a quadrilateral mesh as $n=2$, we consider the trial space $\XX$ of piecewise constants with respect to $\PP$.
In particular, this allows us to perform integration in time analytically for all integrals that are involved in the computation of the Galerkin matrix and the evaluation of the single-layer operator $\mathscr{V}$; see, e.g., \cite{costabel90}.
The remaining integrals over $\Gamma$ have a logarithmic singularity, for which we use the quadrature rules from~\cite{smith00}.
For the computation of the Sobolev--Slobodeckij seminorm in the Faermann estimator $\eta_\PP(\Phi)$,  we use Duffy transformations and Gauss quadrature for the regularized integrands.
Further details on the numerical computation of the involved singular integrals are found in Appendix~\ref{sec:computation}.
The source code that we used to generate the numerical results is available at~\cite{gvv21}.

\subsection{Adaptive algorithm}
In our numerical experiments below, we employ the following adaptive algorithm with $\theta = 0.9$.

\begin{algorithm}\label{alg:main}
Let $0<\theta\le1$ be a marking parameter and $\PP = \set{J\times K}{J\in \PP_{\overline I}, K\in\PP_\Gamma}$ be an initial tensor-mesh corresponding to a mesh $\PP_\Gamma$ of $\Gamma$  and a mesh $\PP_{\overline I}$ of $\overline I=[0,T]$.
For each $\ell = 0, 1, 2, \dots$, iterate the following steps:
\begin{itemize}
 \item[\rmfamily(i)] Compute Galerkin approximation $\Phi_\ell$ of $\phi$ in the space $\XX_\ell$ of all $\PP_\ell$-piecewise constant functions on $\Sigma$.
\item[\rmfamily(ii)] Compute indicators $\eta_{\PP_\ell}^\x(\Phi_\ell,J\times K)$ and $\eta_{\PP_\ell}^t(\Phi_\ell,J\times K)$ for all elements $J\times K\in\PP_\ell$.
\item[\rmfamily(iii)] Determine two minimal sets of marked elements $\MM_\ell^\x,\MM^t_\ell\subseteq\PP_\ell$ such that
\begin{align}\label{eq:marking}
 \theta^2 \eta_{\PP_\ell}(\Phi_\ell)^2 \le  \sum_{J \times K \in \mathcal{M}^\x_\ell} \eta^\x_{\PP_\ell}(\Phi_\ell, J \times K)^2 + \sum_{J \times K \in \mathcal{M}^t_\ell} \eta^t_{\PP_\ell}(\Phi_\ell, J \times K)^2.
\end{align}
\item[\rmfamily(iv)] Refine at least all marked elements to obtain a new mesh $\PP_{\ell + 1}$.
\end{itemize}
\end{algorithm}

We will focus on \emph{isotropic} and \emph{anisotropic} adaptive strategies:
\begin{itemize}
\item In isotropic refinement, we require $\MM_\ell^\x = \MM_\ell^t$ in the marking step {\rm (iii)}, so that \eqref{eq:marking} simplifies to
$\theta^2 \eta_{\PP_\ell}(\Phi_\ell)^2 \le \sum_{J \times K \in \mathcal{M}_\ell} \eta_{\PP_\ell}(\Phi_\ell, J \times K)^2$.
In the refinement step~{\rm (iv)}, we iteratively mark additional elements to ensure that, after subdividing all marked elements into four congruent rectangles, the new mesh $\PP_{\ell+1}$ has only one hanging node per edge.
\item In anisotropic refinement, we bisect the elements $\MM_\ell^\x \setminus \MM_\ell^t$ in space, bisect the elements $\MM_\ell^t \setminus \MM_\ell^\x$ in time, and subdivide all elements $\MM_\ell^\x \cap \MM_\ell^t$ into four congruent rectangles. Then, we iteratively bisect additional elements in space and/or time to ensure that the level difference in space and in time between elements sharing an edge in the new mesh $\PP_{\ell+1}$ is bounded by $1$.
Here, the level in space and the level in time of elements  are defined as the number of bisections in space and time, respectively, to obtain the element from the initial mesh $\PP_0$.
\end{itemize}
For comparison, we also include uniform refinement, where $\PP_{\ell+1}$ is obtained from $\PP_\ell$ by subdividing each element into four congruent rectangles.
For all considered refinement strategies, it is easy to see that the mesh constants from~\eqref{eq:neighbors}--\eqref{eq:local quasi-uniformity} corresponding to $(\PP_\ell)_{\ell\in\N_0}$ depend only on the initial mesh $\PP_0$.

\subsection{Reference for exact error}
As the exact error $\norm{\phi-\Phi}{H^{-1/2,-1/4}(\Sigma)}$ cannot be readily computed in the examples below, we compare the error estimator $\eta_\PP$ and the weighted $L_2$-terms $\zeta_\PP$ from Section~\ref{sec:estimators} with the following $(h-h/2)$-estimator:
For a mesh $\PP$, define the uniformly refined mesh as $\widehat{\PP}$.
With the  the Galerkin approximation $\widehat{\Phi}$ from the refined trial space, we define the $(h-h/2)$-estimator as
\[
    \|\Phi - \widehat{\Phi}\|_{\mathscr{V}}^2 := \dual{ \mathscr{V}(\Phi - \widehat{\Phi})}{\Phi - \widehat{\Phi}}_\Sigma.
\]
Under the saturation assumption $\|\phi - \widehat{\Phi}\|_{\mathscr{V}} \le q_{\rm sat} \|\phi - \Phi\|_{\mathscr{V}}$, the triangle inequality shows that this estimator is equivalent to $\|\phi - \Phi\|_{\mathscr{V}}$, and therefore to the error $\|\phi - \Phi\|_{H^{-1/2,-1/4}(\Sigma)}$ by coercivity of $\mathscr{V}$.
Note that the saturation assumption is indeed satisfied under the realistic (asymptotic) assumption that $\|\phi - \Phi\|_{\mathscr{V}} = \mathcal O\left((\#\PP)^{-s}\right)$ for some arbitrary rate $s>0$.

\subsection{Smooth problem}\label{sec:smooth}
Let $\Omega = (0,1)^2$ and consider the (smooth) solution $u(t,x_1,x_2) := \exp(-2\pi^2 t) \sin(\pi x_1) \sin(\pi x_2)$ with initial condition $u_0(x_1,x_2) := \sin(\pi x_1) \sin (\pi x_2)$ and Dirichlet data $u_D \equiv 0$.
We choose $\PP_0:=\set{[0,1]\times K}{K\in\PP_\Gamma}$ with the uniform mesh $\PP_\Gamma$ of $\Gamma$ being aligned with the corners and consisting of four elements,
as initial mesh of the space-time boundary $\Sigma$.

Figure~\ref{fig:smooth} displays the results in double-logarithmic plots so that the slopes of the lines indicate the corresponding convergence rates.
With the number of degrees of freedom $N=\#\PP$, we see the expected rate $\OO(N^{-5/8})=\OO(N^{-0.625})$ from~\eqref{eq:rates} for both uniform refinement and isotropic refinement (with still slightly worse rate for the weighted $L_2$-terms $\zeta_\PP(\Phi)$ for uniform refinement), albeit adaptive isotropic refinement offers quantitively  better results. For anisotropic refinement refinement, the rate is improved to $\OO(N^{-15/22}) \approx \OO(N^{-0.68})$.
According to \eqref{eq:rates}, this coincides with  the best possible rate that can be achieved with uniform tensor-meshes, where the optimal scaling parameter in $h_t \eqsim h_\x^\sigma$ is given by $\sigma=6/5$.
Note that we do not require setting an explicit scaling in our anisotropic adaptive algorithm, it recovers the optimal rate automatically.

\begin{figure}
\centering
\includegraphics[width=\linewidth]{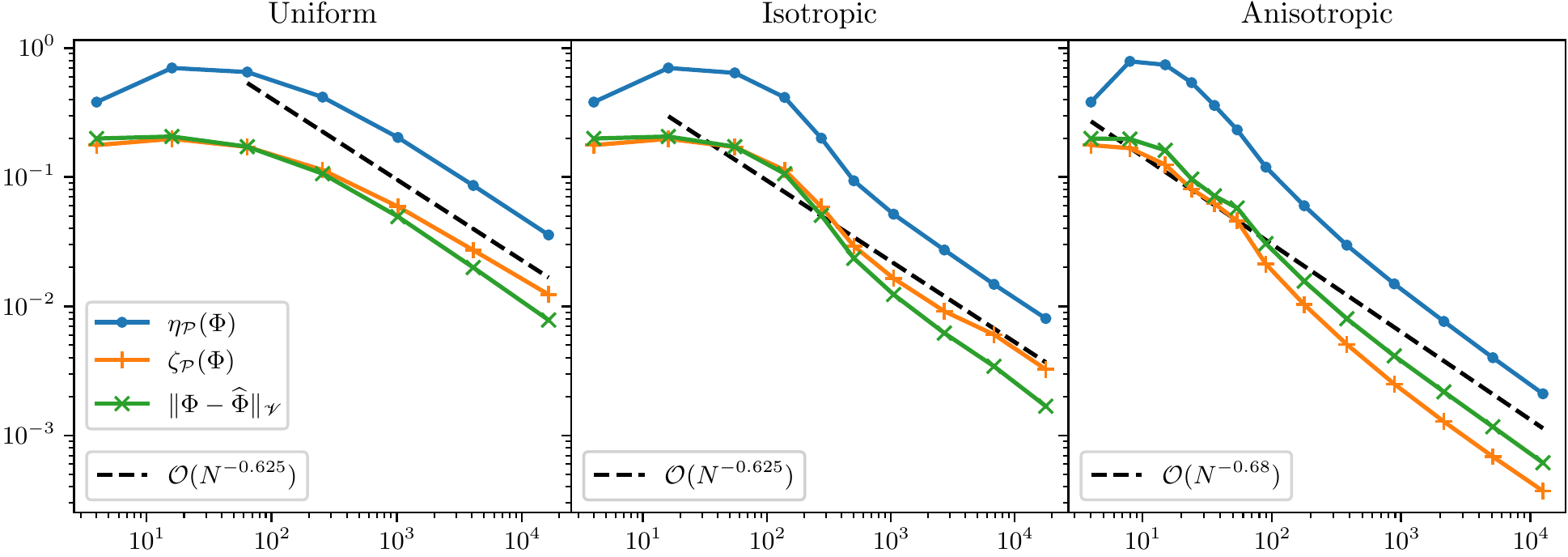}
\vspace{-1em}
\caption{
Error estimators for the smooth problem of Section~\ref{sec:smooth} plotted double-logarithmically over the degrees of freedom $N=\#\PP$:  uniform refinement (left), isotropic refinement (middle), and anisotropic refinement (right).
}\label{fig:smooth}
\end{figure}

\subsection{Mildly singular problem}\label{sec:mildly singular}
Let $\Omega = (0,1)^2$, with initial condition $u_0 \equiv 0$ and Dirichlet data $u_D(t,x_1,x_2) := t^2$. We expect the solution here to be only singular in the four corners of the unit square as the
initial condition is compatible with the Dirichlet data.
The initial mesh $\PP_0$ is chosen as in Section~\ref{sec:smooth}.
Figure~\ref{fig:mildly singular} displays the results. The assymptotic decay rate for all estimators under uniform refinement seems to be $\OO(N^{-1/3})$,
which is improved to $\OO(N^{-1/2})$ for isotropic refinement, and finally, under anistriopic refinement this becomes the optimal rate $\OO(N^{-15/22})$.

\begin{figure}
\centering
\includegraphics[width=\linewidth]{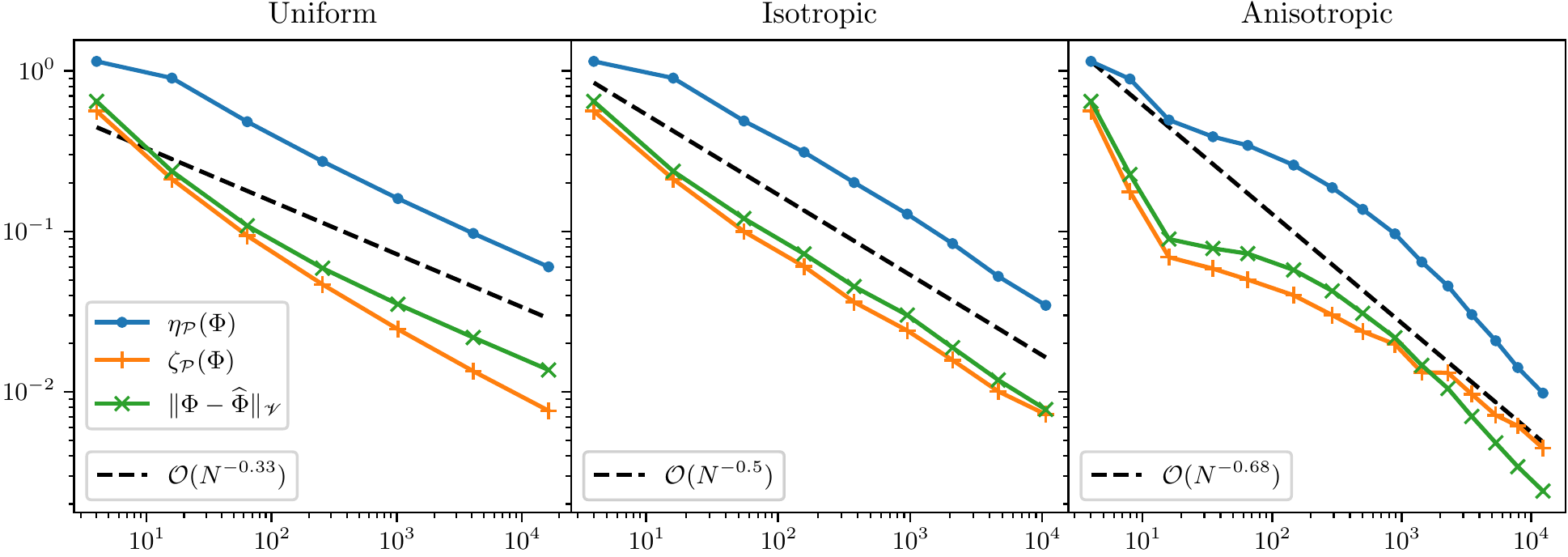}
\vspace{-1em}
\caption{Error estimators for the mildly singular problem of Section~\ref{sec:mildly singular} plotted double-logarithmically over the degrees of freedom $N=\#\PP$:  uniform refinement (left), isotropic refinement (middle), and anisotropic refinement (right).}\label{fig:mildly singular}
\end{figure}

\subsection{Singular problem}\label{sec:singular}
Let $\Omega = (0,1)^2$ with initial condition $u_0 \equiv 0$ and Dirichlet data $u_D \equiv 1$. The solution to this problem is known to have
a strong singularity for $t  = 0$ due to the incompatibility of initial and boundary conditions, in addition to singularities in the four corners of the unit square.
The initial mesh $\PP_0$ is chosen as in Section~\ref{sec:smooth}.

Figure~\ref{fig:singular} displays the results.
The Faermann estimator $\eta_{\PP}(\Phi)$ and the $(h-h/2)$-estimator $\|\Phi - \widehat{\Phi}\|_{\mathscr{V}}$ show both the same sensible convergence behavior for this problem. For uniform refinement, they display the rate $\OO(N^{-1/8})$,
which is then improved by isotropic refinement to $\OO(N^{-1/4})$. Finally, for anistropic refinement, they achieve the best possible rate $\OO(N^{-15/22})$, recovering the rate for a smooth problem.
Looking at Figure~\ref{fig:mesh}, we see strong anisotropic refinement towards $t=0$ with elements of size $h_\x = 1, h_t =2^{-18}$, and some mild refinement towards the corners of the unit square.

On the other hand, the weighted $L_2$-terms $\zeta_{\PP}(\Phi)$ do not seem to decay for uniform or isotropic refinement, and seem to degenerate for anisotropic refinement. This is problematic for
the reliability bound in Corollary~\ref{cor:main}. Further inspection suggests that this is a theoretical shortcoming rather than a practical one.
This is hinted by the $(h-h/2)$-estimator, which one generally assumes to be reliable.
Note that this does not contradict the theoretical results from Corollary~\ref{cor:main}, which states $\zeta_\PP(\Phi)\lesssim\eta_\PP(\Phi)\lesssim \norm{\phi-\Phi}{H^{-1/2,-1/4}(\Sigma)}$ only under the additional  parabolic scaling assumption $h_t \eqsim h_\x^2$ for all space-time elements.

Under this parabolic scaling assumption, the optimal error decay rate for smooth problems  becomes $\OO(N^{-1/2})$; see~\eqref{eq:rates}. Figure~\ref{fig:singular-graded} displays the results
of uniform and adaptive refinement, with meshes that satisfy this scaling constraint\footnote{
For uniform refinement, all elements are bisected once in space-direction and three times in time-direction.
For adaptive refinement, we assume $\MM_\ell^\x=\MM_\ell^t$. All these marked elements are subdivided into four congruent rectangles, where we use additional bisections in space and/or time to guarantee that the level differences between elements sharing an edge is bounded by $1$ and that $\tfrac{1}{2} h_\x^2\le h_t \le 2h_\x^2$.}, providing convergence rates $\OO(N^{-0.18})$ and $\OO(N^{-0.4})$, respectively, for \emph{all} considered estimators.

\begin{figure}
\centering
\includegraphics[width=\linewidth]{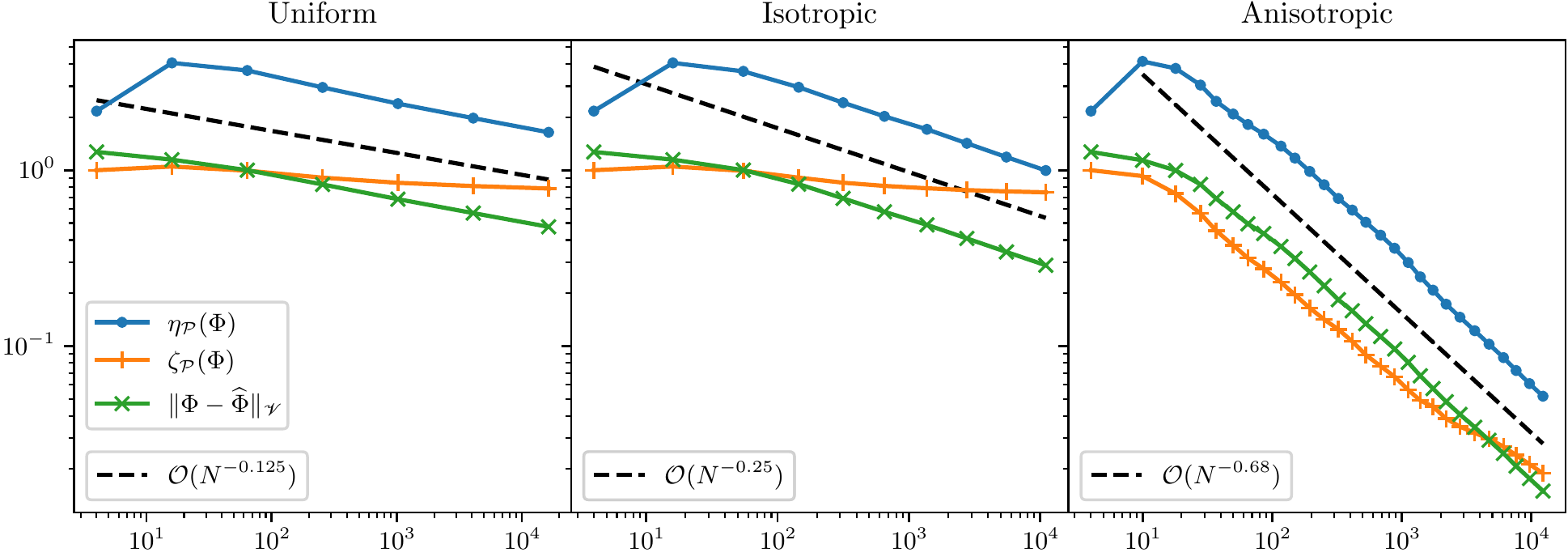}
\vspace{-1em}
\caption{
Error estimators for the singular problem of Section~\ref{sec:singular} plotted double-logarithmically over the degrees of freedom $N=\#\PP$:  uniform refinement (left), isotropic refinement (middle), and anisotropic refinement (right).
}\label{fig:singular}
\end{figure}
\begin{figure}
      \centering
  \begin{minipage}[t!]{0.38\linewidth}
    \includegraphics[width=\linewidth]{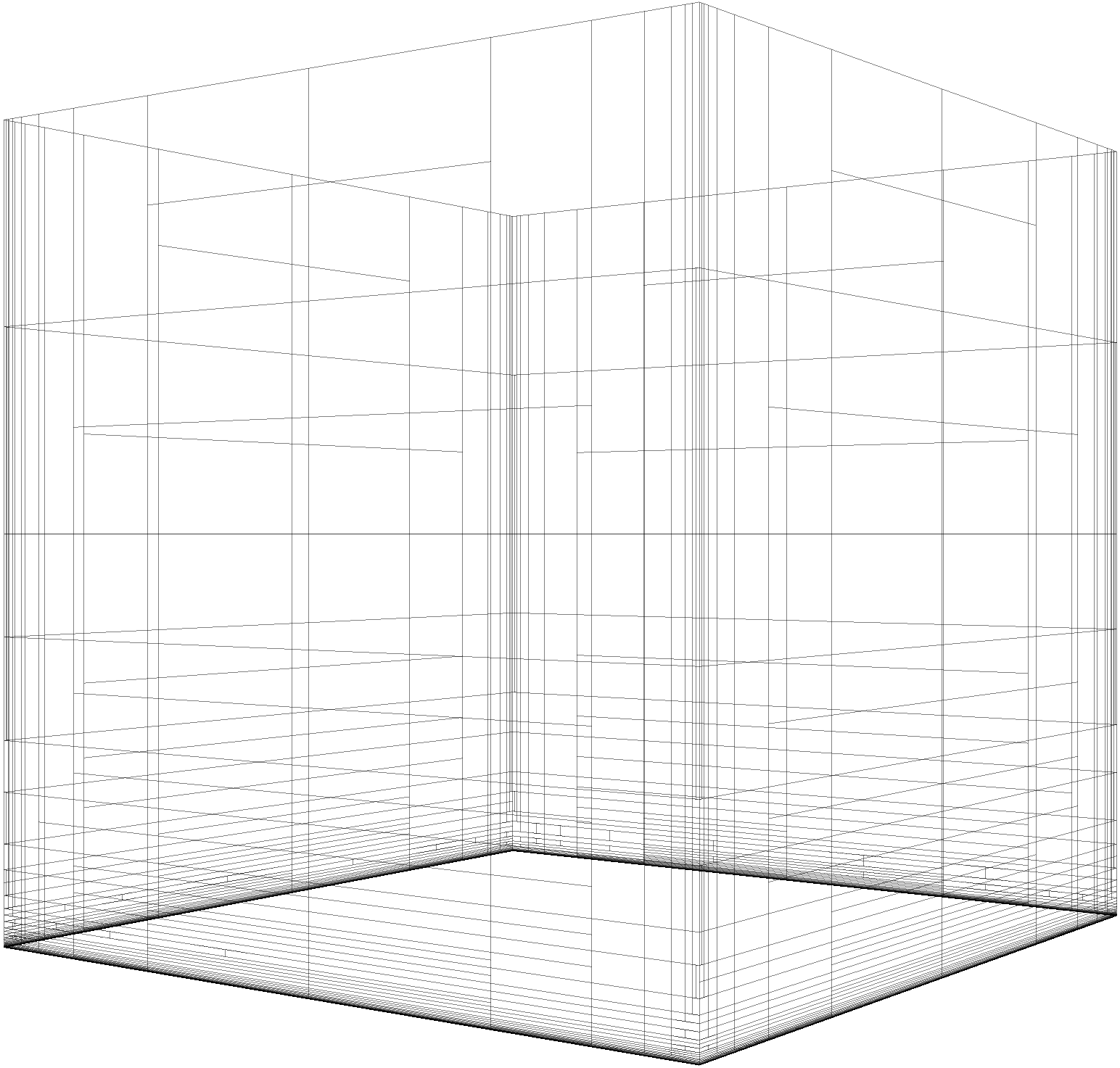}
      \captionsetup{width=\linewidth}
    \caption{Mesh with $N = 1391$ elements, generated by anisotropic refinement for the singular problem of Section~\ref{sec:singular}.}\label{fig:mesh}
  \end{minipage}
  \hfill
  \begin{minipage}[t!]{0.59\linewidth}
    \includegraphics[width=\linewidth]{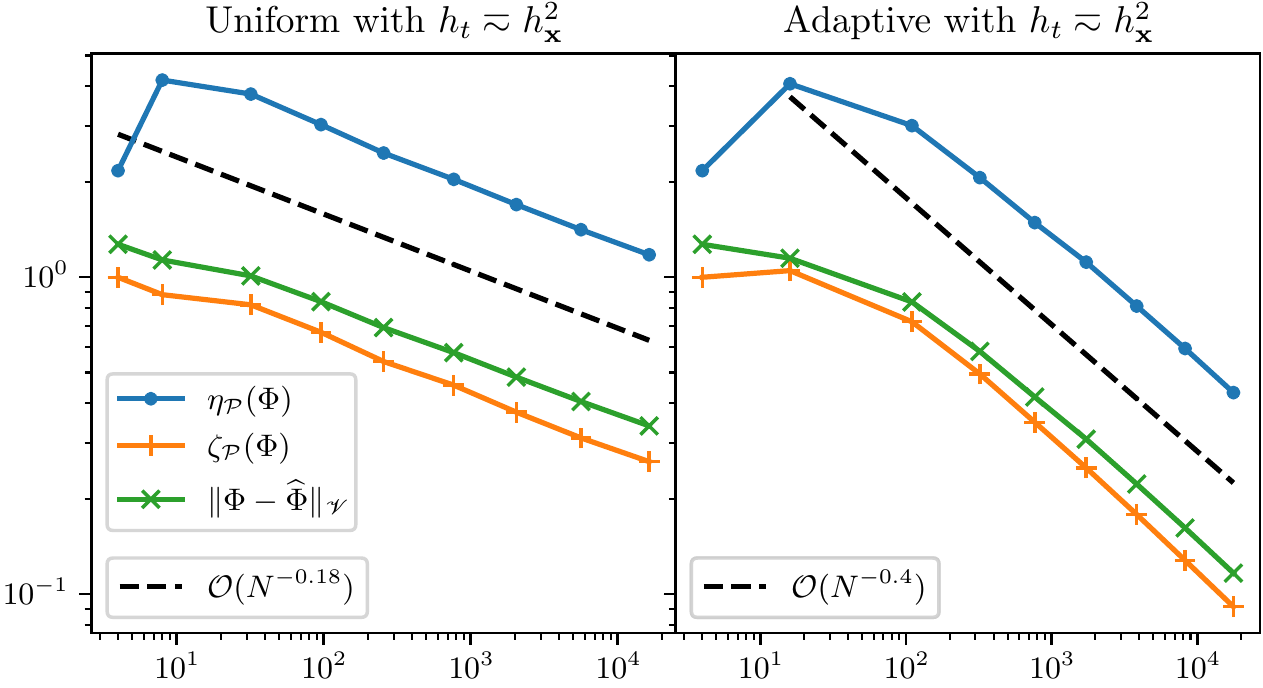}
    \caption{
Error estimators for the singular problem of Section~\ref{sec:singular} plotted double-logarithmically over the degrees of freedom $N=\#\PP$:  uniform refinement (left) and adaptive refinement (right) with parabolic scaling $h_t\eqsim h_\x^2$.} \label{fig:singular-graded}
  \end{minipage}
\end{figure}

\subsection{Singular L-shape problem}\label{sec:singular-lshape}
We consider the L-shaped domain $\Omega := (-1,1)^2 \setminus [-1,0]^2$ with data $u_0 \equiv 1$
and $u_D \equiv 0$. The solution has a strong singularity for $t = 0$, in addition to a singularity at the re-entrant corner $(0,0)$.
We choose $\PP_0:=\set{[0,1]\times K}{K\in\PP_\Gamma}$, with the uniform mesh $\PP_\Gamma$ of $\Gamma$ being aligned with the corners and consisting of eight elements,
as initial mesh of the space-time boundary $\Sigma$.
Figure~\ref{fig:singular-lshape} displays the results, which are similar to those of Section~\ref{sec:singular} with a better behavior of the weighted $L_2$-terms $\zeta_\PP(\Phi)$ for anisotropic refinement.

Enforcing the parabolic scaling $h_t\eqsim h_\x^2$ as in Section~\ref{sec:singular}, \emph{all} estimators converge again with the same rates, being $\mathcal O(N^{-0.18})$ for uniform refinement and $\mathcal O(N^{-0.45})$ for adaptive refinement (not displayed).

\begin{figure}
\centering
\includegraphics[width=\linewidth]{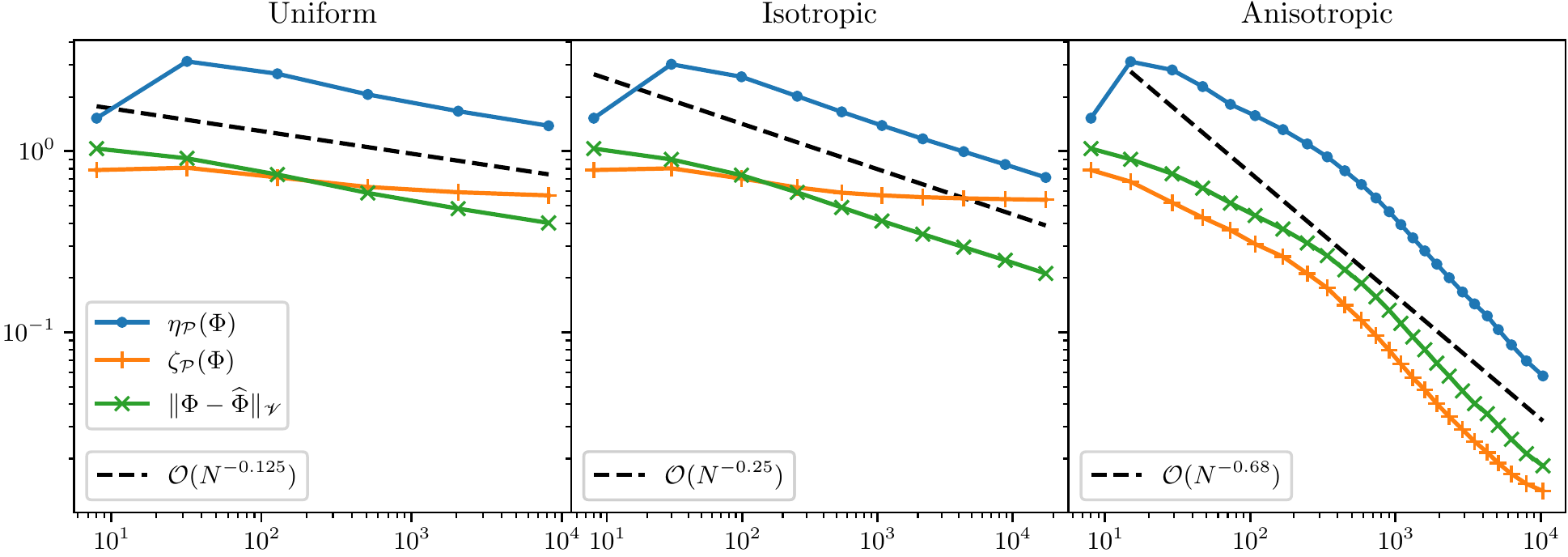}
\vspace{-1em}
\caption{Error estimators for the singular L-shape problem of Section~\ref{sec:singular-lshape} plotted double-logarithmically over the degrees of freedom $N=\#\PP$:  uniform refinement (left), isotropic refinement (middle), and anisotropic refinement (right).
}
\label{fig:singular-lshape}
\end{figure}

\appendix
\section{Numerical computation}\label{sec:computation}
Let $\Omega\subset\R^2$ be a simply connected Lipschitz domain and $\gamma:[0,L]\to \Gamma$ be a parametrization of its boundary $\Gamma=\partial\Omega$.
For a given prismatic mesh $\PP$, i.e., a quadrilateral mesh, of the space-time boundary $\Sigma$, we briefly explain how to numerically compute the corresponding Galerkin solution $\Phi$ of~\eqref{eq:Galerkin} in the trial space of piecewise constants with respect to $\PP$ as well as the corresponding error estimator $\eta_\PP(\Phi)$ and the weighted $L_2$-terms $\zeta_\PP(\Phi)$.
For all $t\in \overline I$, we assume that $\gamma$ is piecewise smooth with respect to the corresponding spatial mesh $\PP|_t$, where for simplicity $|\gamma'|=1$.

We start with the following analytic observations which will be used to integrate the involved integrals in time:
With the exponential integral $\Ei(x):=-\int_{-x}^\infty y^{-1} e^{-y} \d y$ and
\begin{align*}
 \mathfrak{g}_t(\x) := \begin{cases}\frac{1}{4\pi} \Ei\big(-\frac{|\x|^2}{4t}\big) \quad&\text{for }(t,\x)\in (0,\infty)\times\R^{2}, \\ 0 \quad&\text{else},\end{cases}
\end{align*}
it holds that
\begin{align}\label{eq:one time integral}
 \int_{\tilde a}^{\tilde b} G(t-s,\x) \d s = \mathfrak{g}_{t-\tilde b} (\x) - \mathfrak{g}_{t-\tilde a}(\x)
 \quad \text{for all }0\le \tilde a<\tilde b\text{ and }(t,\x)\in[0,\infty)\times\R^{2}.
\end{align}
With
 \begin{align*}
 \mathfrak{G}_t(\x) := \begin{cases}\frac{1}{4\pi} \Big(t e^{-\frac{|\x|^2}{4t}} + t(1+\frac{|\x|^2}{4t})\Ei\big(-\frac{|\x|^2}{4t}\big)\Big) \quad&\text{for }(t,\x)\in (0,\infty)\times\R^{2},
 \\ 0
 \quad&\text{else},\end{cases}
\end{align*}
it further holds that
\begin{align}\label{eq:two time integrals}
 \int_a^b  \int_{\tilde a}^{\tilde b} G(t-s,\x) \d s \d t = \mathfrak{G}_{b-\tilde b}(\x) - \mathfrak{G}_{b-\tilde a}(\x) + \mathfrak{G}_{a-\tilde a}(\x)  - \mathfrak{G}_{a-\tilde b}(\x);
\end{align}
for all $0\le a<b$, $0\le \tilde a<\tilde b$, and $\x\in\R^2$;
see, e.g., \cite{costabel90} or \cite{reinarz15} for more details.
As $\Ei(\cdot)-\log|\cdot|$ is smooth, $\mathfrak{g}_t$ as well as $\mathfrak{G}_t$ have a logarithmic singularity for $\x\to 0$ (provided they are not identically $0$).

\subsection{Galerkin discretization}
To compute the Galerkin discretization~\eqref{eq:Galerkin} for  the trial space $\XX$ of piecewise constants with respect to $\PP$, we have to compute
\begin{align*}
 \dual{\mathscr{V} \1_{\tilde J\times \tilde K}} {\1_{J\times K}}_\Sigma \quad \text{and}\quad \dual{f}{\1_{J\times K}}_\Sigma \quad \text{for all }J\times K, \tilde J\times \tilde K\in\PP.
\end{align*}
Let $J=[a,b]$, $\tilde J = [\tilde a, \tilde b]$, $K=\gamma([c,d])$, and $\tilde K=\gamma([\tilde c,\tilde d])$.
We abbreviate $\gamma_K(\hat x) := \gamma(c+\hat x(d-c))$ and $\gamma_{\tilde K} := \gamma(\tilde c + \hat y (\tilde d - \tilde c))$.

\subsubsection{Galerkin matrix}\label{sec:galerkin matrix}
The Fubini theorem and \eqref{eq:two time integrals} show that
\begin{align*}
  \dual{\mathscr{V} \1_{\tilde J\times \tilde K}} {\1_{J\times K}}_\Sigma = \int_K\int_{\tilde K} \mathfrak{G}_{b-\tilde b}(\x-\y) - \mathfrak{G}_{b-\tilde a}(\x-\y) + \mathfrak{G}_{a-\tilde a}(\x-\y)  - \mathfrak{G}_{a-\tilde b}(\x-\y) \d\y\d\x.
\end{align*}
To compute terms of the form $\int_K\int_{\tilde K} \mathfrak{G}_{t}(\x-\y)\d\y\d\x$ for $t>0$, we first use the transformation formula
\begin{align*}
 \int_K\int_{\tilde K}  \mathfrak{G}_t(\x-\y) \d\y\d\x = (d-c) (\tilde d - \tilde c) \int_0^1\int_0^1 \mathfrak{G}_t\big(\gamma_K(\hat x)-\gamma_{\tilde K}(\hat y)\big) \d\hat y\d\hat x.
\end{align*}

If $K\cap \tilde K=\emptyset$, the integrand $F(\hat x,\hat y):=\mathfrak{G}_t(\gamma_K(\hat x)-\gamma_{\tilde K}(\hat y))$ is smooth and we can use standard Gauss quadrature in both directions.

If $K\cap \tilde K\neq\emptyset$, we assume without loss of generality that $K=\tilde K$ or that $K$ and $\tilde K$ intersect only in one point, i.e., $\#(K\cap \tilde K) = 1$, otherwise we can just split $K$ and $\tilde K$.

If $K=\tilde K$, the integrand has a logarithmic singularity along the diagonal $\hat x=\hat y$.
More precisely, it has the form $F(\hat x,\hat y)=f_1(\hat x,\hat y) + f_2(\hat x,\hat y) \log(|\hat x - \hat y|^2)$ with smooth functions $f_1$ and $f_2$.
We employ the Duffy transformations $\tau_1(\hat x,\hat y) := (\hat x,\hat x\hat y )$ and $\tau_2(\hat x, \hat y) := (\hat x \hat y, \hat x)$, which both map the (open) unit square bijectively onto some (open) triangle, where $[0,1]^2=\bigcup_{i=1}^2\tau_i([0,1]^2)$ with intersection of measure zero between the sets.
As $|\det (D \tau_i(\hat x,\hat y))| = \hat x$ for  $i=1,2$, the integral can be written as
\begin{align*}
 \int_0^1\int_0^1
 F(\hat x,\hat y,\hat z) \d \hat y \d \hat x
 = \int_0^1 \int_0^1 \hat x \sum_{i=1}^2 F\big(\tau_i(\hat x,\hat y)\big) \d \hat y \d \hat x.
\end{align*}
The final integrand of the form $\tilde f_1(\hat x,\hat y) + \tilde f_2(\hat x,\hat y) \log(\hat x) +\tilde f_3(\hat x,\hat y)\log(\hat y)$ with smooth functions $\tilde f_1$, $\tilde f_2$, and $\tilde f_3$, 
 and we can use the quadrature from \cite{smith00} in both $\hat x$- and $\hat y$-direction.

Finally, if $\#(K\cap \tilde K) = 1$, the integrand has a logarithmic singularity at $(\hat x,\hat y)=(0,1)$ or $(\hat x,\hat y)=(1,0)$.
Without loss of generality, we suppose that the singularity is at $(\hat x,\hat y)=(0,1)$ so that the integrand has the form $F(\hat x,\hat y)=f_1(\hat x,\hat y) + f_2(\hat x,\hat y) \log(|1+ \hat x - \hat y|^2)$ with smooth functions $f_1$ and $f_2$.
We rotate the integration domain by $\pi/2$, i.e., $(\hat x,\hat y)\mapsto (\hat y,1-\hat x)$, which transforms the singularity to $(\hat x,\hat y)=(0,0)$, and then employ the same transformations as for the case $K=\tilde K$.
Similar as before, the final integrand is of the form $\tilde f_1(\hat x,\hat y) + \tilde f_2(\hat x,\hat y) \log(\hat x)$ with smooth functions $\tilde f_1$ and $\tilde f_2$, 
and we can use the quadrature from \cite{smith00} in $\hat x$-direction and standard Gauss quadrature in $\hat y$-direction.

We remark that $\int_K\int_{\tilde K} \mathfrak{G}_{t}(\x-\y)\d\y\d\x$ can even be computed exactly if $K$ and $\tilde K$ lie both on \emph{one} straight line.

\subsubsection{Right-hand side vector}\label{sec:rhs vector}
We consider the indirect boundary element method from Section~\ref{sec:integral equations} so that $f=u_D-\mathscr{M}_0 u_0$.
Provided that $u_D$ is $\PP$-piecewise smooth, the term $\dual{u_D}{\1_{J\times K}}_\Sigma$ can be easily computed by first transforming $J\times K$ onto $[0,1]^2$ and subsequently applying Gauss quadrature in both directions.
For $\mathscr{M}_0 u_0$ we employ the Fubini theorem and~\eqref{eq:one time integral},
\begin{align*}
 \dual{\mathscr{M}_0 u_0}{\1_{J\times K}} = \int_K\int_\Omega \big(\mathfrak{g}_a(\x-\y) - \mathfrak{g}_b(\x-\y)\big) u_0(\y)\d\y\d\x.
\end{align*}
The integrand has a logarithmic singularity for $\y\in\partial\Omega$.

Let $\TT_K$ be a partition of $\Omega$ into curvilinear triangles of the form $T=\gamma_T(\hat T)$
with the reference triangle $\hat T=\set{(\hat y,\hat z)\in[0,1]^2}{\hat z\le1- \hat y}$ and some smooth diffeomorphism $\gamma_T:\hat T\to T$ such that for all $T,\tilde T\in\TT_K$ with $T\neq \tilde T \in\TT_K$, the intersection has measure zero.
Moreover, we suppose that there exists a unique element $T\in\TT_K$ with $K\cap T= K$.
With $t\in\{a,b\}$ and the abbreviation $\tilde u_{0,T}:=(u_0\circ\gamma_T) |\det (D\gamma_T)|$, we have that
\begin{align*}
\int_K\int_\Omega \mathfrak{g}_t(\x-\y) u_0(\y) \d\y\d\x
= (d-c) \sum_{T\in\TT_K} \int_0^1\int_{\hat T} \mathfrak{g}_t\big(\gamma_K(\hat x)-\gamma_T(\hat y,\hat z)\big) \, \tilde u_{0,T}(\hat y,\hat z)\d \hat y \d\hat z  \d \hat x.
\end{align*}

\begin{remark}
To construct $\TT_K$ in our examples from Section~\ref{sec:numerics}, we start from some initial mesh of $\Omega$ consisting of one square for $\Omega=(0,1)^2$ and three squares for $\Omega := (-1,1)^2 \setminus [-1,0]^2$, and proceed as follows:
    First, we dyadically refine the element that contains $K$ until $K$ becomes the edge of one of the resulting squares.
We use further dyadic refinements to ensure that there is at most one hanging node per edge.
To obtain a triangular mesh $\TT_K$, we finally bisect the elements along one diagonal.
Note that the resulting $\TT_K$ is \emph{not} conforming.
The number of elements in $\TT_K$ is proportional to the level of $K$.
\end{remark}

If $K\cap T=\emptyset$, the integrand $F(\hat x,\hat y,\hat z):=\mathfrak{g}_t(\gamma_K(\hat x)-\gamma_T(\hat y,\hat z)) \, \tilde u_{0,T}(\hat y,\hat z)$ is smooth and we can use standard Gauss quadrature in all three directions.

If $K\cap T=K$, we suppose that $\gamma_K=\gamma_T(\cdot,0)$ so that the integrand has a logarithmic singularity for $(\hat x,\hat z)=(\hat y,0)$.
More precisely, it has the form
\begin{align*}
F(\hat x,\hat y,\hat z)=f_1(\hat x,\hat y,\hat z) + f_2(\hat x,\hat y,\hat z)\log(|{\bf F}(\hat x,\hat y,\hat z)(\hat x-\hat y, \hat z)^\top|^2)
\end{align*}
 for some smooth functions $f_1$, $f_2$ with values in $\R$, and ${\bf F}$ with values in $\R^{2\times 2}$ and $\det {\bf F}\neq0$.
We employ the following transformations
\begin{align*}
 \tau_1(\hat x,\hat y,\hat z) &:= \big(\hat x,\hat x(1-\hat y),\hat x\hat y\hat z \big),
 \\
 \tau_2(\hat x,\hat y,\hat z) &:= \big(\hat x(1-\hat y), \hat x(1-\hat y\hat z),\hat x\hat y \hat z \big),
 \\
 \tau_3(\hat x,\hat y,\hat z) &:= \big(\hat x(1-\hat y +\hat y\hat z), \hat x (1-\hat y),\hat x\hat y\big),
\end{align*}
which all map the (open) unit cube bijectively onto some (open) tetrahedron, where $[0,1]\times \hat T=\bigcup_{i=1}^3\tau_i([0,1]^3)$ with intersection of measure zero between the sets.
As $|\det (D \tau_i(\hat x,\hat y,\hat z))| = \hat x^2 \hat y$ for  $i=1,2,3$, the integral can be written as
\begin{align*}
 \int_0^1\int_{\hat T}
F(\hat x,\hat y,\hat z) \d \hat y \d\hat z  \d \hat x
 = \int_0^1 \int_0^1 \int_0^1 \hat x^2 \hat y \sum_{i=1}^3 F\big(\tau_i(\hat x,\hat y,\hat z)\big) \d \hat y \d \hat z \d \hat x.
\end{align*}
Note that the vector $(\hat x-\hat y,\hat z)^\top$ from the definition of $F$ is transformed under $\tau_1$, $\tau_2$, and $\tau_3$ to $\hat x\hat y(1,z)^\top$, $\hat x\hat y(z-1,z)^\top$, and $\hat x\hat y(z,1)^\top$, respectively.
We infer that we can use the quadrature from \cite{smith00} in $\hat x$- and $\hat y$-direction, and standard Gauss quadrature in $\hat z$-direction.

If $\#(K\cap T)=1$, we suppose that $\gamma_K(0)=\gamma_T(0,0)$.
We further suppose that $K$ and $T$ can be parametrized via \emph{one} smooth diffeomorphism $\gamma_{K\cup T}$:
The parameter domain of $\gamma_{K\cup T}$ should contain at least the line $\set{\hat x (\hat v,\hat w)}{\hat x \in [0,1]}$ for some $(\hat v,\hat w)\in\R^2\setminus [0,\infty)^2$ and $\hat T$.
Moreover, $\gamma_K(\hat x)=\gamma_{K\cup T}(\hat x (\hat v,\hat w))$  for $\hat x\in[0,1]$ and $\gamma_T(\hat y,\hat z) = \gamma_{K\cup T}(\hat y,\hat z)$ for $(\hat y,\hat z)\in\hat T$.
Then the integrand has a logarithmic singularity for $(\hat x,\hat y,\hat z)=(0,0,0)$.
More precisely, it has the form
\begin{align*}
F(\hat x,\hat y,\hat z):=f_1(\hat x,\hat y,\hat z) + f_2(\hat x,\hat y,\hat z)\log(|{\bf F}(\hat x,\hat y,\hat z)(\hat x \hat v -\hat y,\hat x \hat w -\hat z)^\top|^2)
\end{align*}
 for some smooth functions $f_1$, $f_2$ with values in $\R$, and ${\bf F}$ with values in $\R^{2\times 2}$ and $\det {\bf F}\neq0$.
Applying the transformations $\tau_i$ from before, the vector $(\hat x \hat v -\hat y,\hat x \hat w -\hat z)^\top$ from the definition of $F$ is transformed to
$\hat x( (\hat v,\hat w)^\top-(1-\hat y, \hat y\hat z)^\top)$,
$\hat x((1-\hat y)(v,w)^\top - (1-yz, yz)^\top)$,
and $\hat x((1-\hat y+\hat y\hat z)(v,w)^\top - (1-y,y)^\top)$, respectively.
Due to our assumption on $(v,w)$, up to the factor $\hat x$, each of these terms is uniformly away from $(0,0)^\top$ for all $(\hat y,\hat z)\in\hat T$.
Thus, we can again use the quadrature from \cite{smith00} in $\hat x$-direction, and standard Gauss quadrature in $\hat y$ and  $\hat z$-direction.

\subsection{Evaluation of residual}
To compute the error estimator $\eta_\PP(\Phi)$ as well as the weighted $L_2$-terms $\zeta_\PP(\Phi)$, we have to evaluate the residual $f-\mathscr{V}\Phi$.

\subsubsection{Single-layer operator}
To evaluate the single-layer operator for piecewise constants with respect to $\PP$, we have to compute
\begin{align*}
 (\mathscr{V} \1_{\tilde J\times \tilde K})(t,\x) = \int_{\tilde J}\int_{\tilde K} G(t-s,\x-\y) \d \y\d s  \quad \text{for all }\tilde J\times \tilde K\in\PP \text{ and all } (t,\x)\in\Sigma.
\end{align*}
Let $\tilde J = [\tilde a, \tilde b]$ and $\tilde K=\gamma([\tilde c,\tilde d])$, and abbreviate again $\gamma_{\tilde K} := \gamma(\tilde c + \hat y (\tilde d - \tilde c))$.
The Fubini theorem and \eqref{eq:one time integral} show that
\begin{align*}
 (\mathscr{V} \1_{\tilde J\times \tilde K})(t,\x) = \int_{\tilde K} \mathfrak{g}_{t-\tilde b}(\x-\y) - \mathfrak{g}_{t-\tilde a}(\x-\y) \d\y.
\end{align*}
To compute terms of the form $\int_{\tilde K} \mathfrak{g}_{s}(\x-\y)\d\y$ for $s>0$, we first use the transformation formula
\begin{align*}
  \int_{\tilde K} \mathfrak{g}_{s}(\x-\y) \d\y = \int_0^1 \mathfrak{g}_s\big(\x-\gamma_{\tilde K}(\hat y)\big) \d\hat y.
\end{align*}

If $\x\not\in \tilde K$, the integrand is smooth and we can use standard Gauss quadrature.

If $\x\in\tilde K$, we assume without loss of generality that $\x=\gamma(\tilde c)$ or $\x=\gamma(\tilde d)$, otherwise we can just split $\tilde K$.
The integrand is of the form $f_1(\hat y) + f_2(\hat y) \log(\hat y)$ with smooth functions $f_1$ and $f_2$,  and we can use the quadrature from \cite{smith00}.

We remark that $\int_{\tilde K} \mathfrak{g}_{s}(\x-\y) \d\y$ can even be computed exactly  if $\x$ and $\tilde K$ lie both on \emph{one} straight line.

\subsubsection{Initial operator}
We consider the indirect boundary element method from Section~\ref{sec:integral equations} so that $f=u_D-\mathscr{M}_0 u_0$.
To evaluate $f$ at $(t,\x)\in J\times K$ with $t>0$ and $J\times K\in\PP$, let $\TT_K$ be again a curvilinear triangulation of $\Omega$ as in Appendix~\ref{sec:rhs vector}.
With the abbreviation $\tilde u_{0,T}:=(u_0\circ\gamma_T) |\det (D\gamma_T)|$, we have that
\begin{align*}
 (\mathscr{M}_0 u_0)(t,\x) = \sum_{T\in\TT_K} \int_0^1\int_0^1 G\big(t,\x-\gamma_T(\hat y,\hat z)\big)\, \tilde u_{0,T}(\hat y,\hat z) \d\hat y\d\hat z.
\end{align*}
As $\x\neq\gamma_T(\hat y,\hat z)$, the integrand is smooth and we can use standard Gauss quadrature in both directions.

\subsection{Error estimator and $L_2$-terms}
Now that we can evaluate the residual $r:=f-\mathscr{V}\Phi$, we explain how to compute the estimator $\eta_\PP(\Phi)$ as well as the weighted $L_2$-terms $\zeta_\PP(\Phi)$.
We assume that $r$ is, at least $\PP$-piecewise, sufficiently smooth.
In particular, $\zeta_\PP(\Phi)$ can be easily computed via the transformation formula and standard Gauss quadrature in both directions.

For $\eta_\PP(\Phi)$, we need to compute terms of the form $|r|_{L_2(J,H^{1/2}(K\cup\tilde K))}$ and $|r|_{H^{1/4}(J\cup\tilde J,L_2(K))}$ with $K\cap \tilde K \neq\emptyset$ and $J\cap \tilde J\neq\emptyset$.

The first term reads as
\begin{align*}
 |r|_{L_2(J,H^{1/2}(K\cup\tilde K))}^2 = \int_J |r(t,\cdot)|_{H^{1/2}(K)}^2 + 2 \int_K\int_{\tilde K} \frac{|r(t,\x)-r(t,\y)|^2}{|\x-\y|^2}\d \y\d \x+ |r(t,\cdot)|_{H^{1/2}(\tilde K)}^2 \d t.
\end{align*}
We consider the integrand for fixed $t$.
The first and last term can be transformed as in the case $K=\tilde K$ of Section~\ref{sec:galerkin matrix} and subsequently be computed by standard Gauss quadrature in both directions.
Similarly, the middle term can be transformed as in the case $\# (K\cap\tilde K)=1$ of Section~\ref{sec:galerkin matrix} and subsequently be computed by standard Gauss quadrature in both directions.
Finally, we use standard Gauss quadrature in $t$-direction for all three terms.

Now, we consider
\begin{align*}
 |r|_{H^{1/4}(J\cup\tilde J,L_2(K))}^2
 = \int_K  |r(\cdot,\x)|_{H^{1/4}(J)}^2 + 2\int_J\int_{\tilde J} \frac{|r(t,\x)-r(s,\x)|^2}{|t-s|^{3/2}}\d s \d t+ |r(\cdot,\x)|_{H^{1/4}(\tilde J)}^2 \d \x.
\end{align*}
The first and last term can be transformed as in the case $K=\tilde K$ of Section~\ref{sec:galerkin matrix}.
With $\gamma_{J}$ defined analogously as $\gamma_K$, this shows for the first one that
\begin{align*}
 |r(\cdot,\x)|_{H^{1/4}({J})}^2
 &=  2|{J}|^2\int_0^1\int_0^1 \frac{|r(\gamma_{J}(\hat{t}),\x)-r(\gamma_{J}(\hat{t}\hat{s}),\x)|^2}{|\gamma_{J}(\hat{t})-\gamma_{J}(\hat{t}\hat{s})|^{3/2}} \,\hat{t} \d \hat{s}\d \hat{t}
 \\
 &= 2|{J}|^{1/2}\int_0^1\int_0^1 \frac{|r(\gamma_{J}(\hat{t}),\x)-r(\gamma_{J}(\hat{t}(1-\hat{s})),\x)|^2}{\hat{s}} \, \hat{s}^{-1/2}\,\hat{t}^{-1/2}  \d \hat{s}\d \hat{t}.
\end{align*}
As $r$ is piecewise smooth, the quotient is smooth and we can use Gauss quadrature with weight $\hat{t}^{-1/2}$ in $\hat t$-direction and with weight $\hat{s}^{-1/2}$ in $\hat s$-direction.
Similarly, the second term can be transformed as in the case $\# (K\cap \tilde K)=1$ of Section~\ref{sec:galerkin matrix} so that for $\tilde J$ left from $J$ and $\gamma_J,\gamma_{\tilde J}$ defined analogously as $\gamma_K,\gamma_{\tilde K}$, we get that
\begin{align*}
 \int_J\int_{\tilde J} \frac{|r(t,\x)-r(s,\x)|^2}{|t-s|^{3/2}}\d s \d t&= |J|\,|\tilde J| \int_0^1\int_0^1 \frac{|r(\gamma_J(\hat s\hat t),\x)-r(\gamma_{\tilde J}(1-\hat t),\x)|^2}{|\gamma_J(\hat s\hat t)-\gamma_{\tilde J}(1-\hat t)|^{3/2}}\hat t\d \hat s \d \hat t
 \\
 &\quad +  |J|\,|\tilde J| \int_0^1\int_0^1 \frac{|r(\gamma_J(\hat t),\x)-r(\gamma_{\tilde J}(1-\hat s\hat t),\x)|^2}{|\gamma_J(\hat t)-\gamma_{\tilde J}(1-\hat s\hat t)|^{3/2}}\hat t\d \hat s \d \hat t.
 \end{align*}
We can apply Gauss quadrature with weight $\hat t^{-1/2}$ in $\hat t$-direction and with weight $1$ in $\hat s$-direction.
Finally, we use the transformation formula and standard Gauss quadrature in $\hat x$-direction for all three terms.

\section*{Acknowledgement} The first author has been supported by the Austrian Science Fund (FWF) under grant J4379-N.
The second author has been supported by the Netherlands Organization for Scientific Research (NWO) under contract.\ no.\ 613.001.652.

\bibliographystyle{alpha}
\bibliography{literature}

\end{document}